\documentclass{amsart}
\usepackage{amssymb}
\usepackage{amsmath}
\usepackage{amsfonts}
\usepackage{hyperref}
\usepackage{yhmath}
\usepackage{framed}
\newtheorem{theorem}{Theorem}[section]
\newtheorem{lemma}[theorem]{Lemma}

\newtheorem{prop}[theorem]{Proposition}

\newtheorem{definition}[theorem]{Definition}
\newtheorem{remark}[theorem]{Remark}

\renewcommand{\l}{\lambda}
\newcommand{\R}{\mathbb{R}}

\begin{document}
\title[Bilinear estimates with applications to Boussinesq equations]{Revisited bilinear Schr\"{o}dinger estimates with applications to generalized Boussinesq equations}

\author{Dan-Andrei Geba and Evan Witz}

\address{Department of Mathematics, University of Rochester, Rochester, NY 14627, U.S.A.}
\email{dangeba@math.rochester.edu}
\address{Department of Mathematics, University of Rochester, Rochester, NY 14627, U.S.A.}
\email{ewitz@ur.rochester.edu}
\date{}

\begin{abstract}
In this paper, our goal is to improve the local well-posedness theory for certain generalized Boussinesq equations by revisiting bilinear estimates related to the Schr\"{o}dinger equation. Moreover, we propose a novel, automated procedure to handle the summation argument for these bounds.  
\end{abstract}


\maketitle

\section{Introduction}

The focus of this article is to develop a local well-posedness\footnote{Here, well-posedness is meant in the Hadamard sense: existence, uniqueness, and continuity of the data-to-solution map in appropriate topologies.} (LWP) theory for the Cauchy problem given by
\begin{equation}
\left\{
\begin{array}{l}
u_{tt}-\Delta u+\Delta^2u\pm\Delta(u^2)=0, \qquad u=u(x,t): \mathbb{R}^n\times I \to \mathbb{R},\\
\\
u(x,0)\,=\,u_0(x),\qquad u_t(x,0)\,=\,u_1(x),\\
\end{array}\right.
\label{main}
\end{equation}
where $0\in I\subseteq\R$ is an open interval and $(u_0,u_1)\in H^s(\R^n)\times H^{s-2}(\R^n)$. The differential equation above belongs to a family of equations called \emph{generalized Boussinesq equations}, with the $1+1$-dimensional version being known as the \emph{``good" Boussinesq equation}. 

In fact, the $1+1$-dimensional Cauchy problem is the best understood so far, with Kishimoto \cite{K13} showing that it is LWP for $s\geq -1/2$ and ill-posed (IP) for $s<-1/2$. This result capped a sustained drive for this problem with contributors like Bona-Sachs \cite{BS}, Linares \cite{L93}, Fang-Grillakis \cite{FG96}, Farah \cite{F09}, and Kishimoto-Tsugawa \cite{KT10}. Thus, our interest here is in investigating the high-dimensional (i.e., $n\geq 2$) case of \eqref{main}, for which, to our knowledge, the only available results are due to Farah \cite{F092} and Okamoto \cite{O17}. 

The former states that \eqref{main} is LWP for  $u_0\in H^s(\R^n)$, $u_1=\Delta \phi$ with $\phi\in H^{s}(\R^n)$, and 
\begin{equation*}
s\geq \max\left\{0, \frac{n-4}{2}\right\}.
\end{equation*}
We make the remark that the index $(n-4)/2$ appears naturally in connection to our problem since, by ignoring the lower order term $\Delta u$, the equation is scale-invariant under the transformation
\begin{equation*}
u\mapsto u_\l(x,t)=\l^{-2}u(\l^{-1} x, \l^{-2}t)
\end{equation*}
and one has
\begin{equation*}
\|u_\l(0)\|_{\dot{H}^s(\R^n)}= \l^{\frac{n-4}{2}-s}\|u_0\|_{\dot{H}^s(\R^n)}.
\end{equation*}
For the second result, Okamoto proved that \eqref{main} is IP for $(u_0,u_1)\in H^s(\R^n)\times H^{s-2}(\R^n)$ when $s<-1/2$, in the sense that norm inflation occurs and, as a consequence, the associated flow map is discontinuous everywhere. Hence, based on this picture, one is naturally led to study what happens in the regime when
\begin{equation*}
-\frac{1}{2}\leq s< \max\left\{0, \frac{n-4}{2}\right\}.
\end{equation*}
In particular, is it the case that \eqref{main} is LWP for $(u_0,u_1)\in H^s(\R^n)\times H^{s-2}(\R^n)$ with $s<0$ when $n\geq 2$? Our main result provides a partial positive answer to this question.

\begin{theorem} 
If $n=2$ or $n=3$, then \eqref{main} is LWP for $(u_0,u_1)\in H^s(\R^n)\times H^{s-2}(\R^n)$ with $-1/4<s<0$. 
\label{th-main}
\end{theorem}

The argument for this theorem is inspired by an approach due to Kishimoto-Tsugawa \cite{KT10} (see also \cite{K13} and \cite{O17}), in which the first step consists in reformulating \eqref{main} as the Cauchy problem for a nonlinear Schr\"{o}dinger equation with initial data in $H^{s}(\R^n)$. This is followed by setting up a contraction scheme for the integral version of this new Cauchy problem, where we use Bourgain functional spaces and corresponding linear and bilinear estimates.  

The structure of the paper is as follows. In the next section, we start by introducing the notation and terminology used throughout the article and by performing the reformulation step. Also there, we detail the contraction scheme and reduce it to the proof of a family of bilinear estimates related to the Schr\"{o}dinger equation. In section \ref{bil-section}, we revisit work by Colliander-Delort-Kenig-Staffilani \cite{CDKS-01} and Tao \cite{T-01} for this type of bounds, provide a unitary framework to tackle them, and derive results in previously unknown scenarios. In the final section, we discuss an innovative, automated method, based on a Python code, to deal with the summation component of the proof for the bilinear estimates, which might also be of independent interest.  

\subsection*{Acknowledgements}
The first author was supported in part by a grant from the Simons Foundation $\#\, 359727$. 


\section{Preliminaries}

\subsection{Notational conventions and terminology}
First, we agree to write $A\lesssim B$ in a certain setting when $A\leq CB$ and $C>0$ is a constant depending only upon fixed parameters which may change from one setting to another. Moreover, we write $A\sim B$ to denote that both $A\lesssim B$ and $B\lesssim A$ are valid. Next, we recall the notations $\langle a\rangle=(1+|a|^2)^{1/2}$ (for any $a\in\R^n$), 
\begin{equation*}
\widehat{z}(\xi)=\int_{\R^n} e^{-ix\xi}\,z(x)\,dx\quad \text{and}\quad \widehat{w}(\xi, \tau) =\int_{\R^n\times \R} e^{-i(x\xi+t\tau)}\,w(x,t)\,dx\,dt,
\end{equation*}
the last two representing the Fourier transform of $z=z(x)$ and the spacetime Fourier transform of $w=w(x,t)$, respectively. Finally, we let $\varphi=\varphi(t)$ denote the classical, smooth cutoff function $\varphi:\R\to\R$ satisfying $\varphi \equiv 1$ on $[-1,1]$ and supp$(\varphi)\subseteq [-2,2]$.

Following this, we define the Sobolev and Bourgain norms\footnote{From here on out, for a functional space $Y$, we write either $Y=Y(\R^n)$ or $Y=Y(\R^n\times \R)$ as the majority of such norms refers to these two particular situations.}
\begin{align}
\label{deh-hs}\|z\|_{H^s(\R^n)}&:=\|\langle \xi\rangle^s\widehat{z}(\xi)\|_{L^2_\xi(\R^n)},\\ 
\label{def-xst}\|w\|_{X^{s,\theta}(\R^n\times \R)}:=&\|\langle \xi\rangle^s\langle \tau-|\xi|^2\rangle^\theta\widehat{w}(\xi,\tau)\|_{L^2_{\xi,\tau}(\R^n\times \R)},
\end{align}
for arbitrary $s$, $\theta\in \R$. For $T>0$, we will also use the truncated norm
\begin{equation*}
\|z\|_{X^{s,\theta}_T}:=\inf_{w=z\, \text{on}\, [0,T]} \|w\|_{X^{s,\theta}}.
\end{equation*}
Working directly with these norms, one can easily prove the classical bound
\begin{equation}
\|w\|_{L^\infty_tH^s_x}\lesssim \|w\|_{X^{s,\theta}}
\label{Sob-X}
\end{equation}
and the inclusion $X^{s,\theta}\subset C(\R, H^s)$, both for all $s\in \R$ and $\theta>1/2$.

\subsection{Reformulation step}
As mentioned in the introduction, we start the argument for Theorem \ref{th-main} by rewriting \eqref{main} in the form of a Cauchy problem for a Schr\"{o}dinger equation. For this purpose, we define as in \cite{KT10} 
\begin{equation*}
v:=u-i(1-\Delta)^{-1}u_t\quad\text{and}\quad v_0:=u_0-i(1-\Delta)^{-1}u_1. 
\end{equation*}
Straightforward calculations reveal that 
\begin{equation}
\left\{
\begin{array}{l}
iv_{t}-\Delta v=H(v, \overline{v}):=\frac{\overline{v}-v}{2}\pm\omega(D)\left(\frac{v+\overline{v}}{2}\right)^2, \quad v=v(x,t): \mathbb{R}^n\times I \to \mathbb{C},\\
\\
v(x,0)\,=\,v_0(x),\\
\end{array}\right.
\label{main-2}
\end{equation}
where $\omega=\omega(D)$ is the spatial multiplier operator with symbol 
\[
\omega(\xi)=\frac{|\xi|^2}{1+|\xi|^2}.
\] 
Moreover, for an arbitrary $T>0$, the map $(u,u_0,u_1)\mapsto (v,v_0)$ from 
\[U:=(C([0,T], H^s)\cap C^1([0,T], H^{s-2}))\times H^s\times H^{s-2}\] 
to 
\[V:=C([0,T]; H^s)\times H^s\]
is Lipschitz continuous.
Conversely, if $v$ and $v_0$ satisfy \eqref{main-2}, then, by letting
\begin{equation*}
u=\frac{v+\overline{v}}{2}, \quad u_0=\frac{v_0+\overline{v_0}}{2}, \quad \text{and}\quad u_1=(1-\Delta)\left(\frac{\overline{v_0}-v_0}{2i}\right),
\end{equation*}
it is easy to check that that $u$, $u_0$, and $u_1$ are all real-valued and they satisfy \eqref{main}. Furthermore, noticing that 
\[
-2iu_t=(1-\Delta)(v-\overline{v}),
\]
one deduces that the map $(v,v_0)\in V\mapsto(u,u_0,u_1)\in U$ is also Lipschitz continuous. Thus, LWP in $H^s\times H^{s-2}$ for \eqref{main} is equivalent to LWP in $H^s$ for \eqref{main-2}. 

\subsection{Setting up the contraction argument and reducing it to the proof of bilinear Schr\"{o}dinger estimates} In proving that \eqref{main-2} is LWP for $v_0\in H^s$, we adopt the standard procedure and, using Duhamel's formula, write its integral version
\begin{equation}
v(t)=\, S(t)v_0-i\int_0^tS(t-t') H(v(t'), \overline{v}(t'))\,dt',
\label{stv}
\end{equation}
for which we set up a contraction argument using suitable $X^{s,\theta}$ spaces. Above, $S(t)=e^{-it\Delta}$ is the propagator for the linear Schr\"{o}dinger equation $iw_t-\Delta w=0$, i.e.,
\begin{equation*}
w(t)=S(t)w(0), \qquad (\forall)\, t\in\R.
\end{equation*}

\begin{remark}
By comparison, Farah \cite{F092} writes the main equation as
\begin{equation*}
u_{tt}+\Delta^2u=\Delta(u\mp u^2)
\end{equation*}
and, using the Fourier transform and Duhamel's formula, derives
\begin{equation*}
\aligned
u(t)=\, &\frac{S(t)+S(-t)}{2}\,u(0)+\frac{S(t)-S(-t)}{-2i\Delta}\,u_t(0)\\
&+\int_0^t\frac{S(t-t')-S(-t+t')}{2i}\,(-u(t')\pm u^2(t'))\,dt'.
\endaligned
\end{equation*}
Following this, he proves LWP for \eqref{main} by running a contraction argument for this integral formulation in functional spaces related to Strichartz-type estimates for the Schr\"{o}dinger group $(S(t))_{t\in\R}$.
\end{remark}

The next statement is our LWP result for \eqref{main-2}, which, as we argued, implies Theorem \ref{th-main}.

\begin{theorem}
For $n=2$ or $n=3$, if $\theta>1/2$, $(\theta-1)/2<s<0$, and $r\geq 1$, then, for any $\|v_0\|_{H^s}\leq r$, there exist $T\sim r^{-4/(2s-n+4)}$ and $v\in X^{s,\theta}_T\cap C([0,T], H^s)$ solving the integral equation \eqref{stv} on $[0,T]$ with the data-to-solution map
\[
v_0\in\{z;\,\|z\|_{H^s}\leq r\}\mapsto v\in C([0,T], H^s)\cap X^{s,\theta}_T
\]
being Lipschitz continuous. Moreover, this solution is unique in the class of $X^{s,\theta}_T\cap C([0,T], H^s)$ solutions for \eqref{stv}.
\label{main-th-2}
\end{theorem}

As is always the case with this type of results, they are the joint outcome of a set of estimates which are used in the context of a contraction scheme. For the above theorem, these bounds are
\begin{align}
\label{zhs}\|z_\l\|_{H^s}\lesssim \l^{\frac{n}{2}-s-2} \|z\|_{H^s},\ &\\
\label{wst-1}\|w\|_{X^{s,\theta-1}}+\|\overline{w}\|_{X^{s,\theta-1}}\lesssim \|w\|&_{X^{s,\theta}},\\
\label{lineq}\left\|\varphi(t)\left(S(t)z-i\int_0^tS(t-t')F(\cdot,t')\,dt'\right)\right\|_{X^{s,\theta}}&\lesssim \|z\|_{H^s} +\|F\|_{X^{s,\theta-1}},
\end{align}
and 
\begin{align}
\label{oubvb}\|\omega_\l(D)(\overline{u}\,\overline{v})\|_{X^{s,\theta-1}}\lesssim \|u\|_{X^{s,\theta}}\|v\|_{X^{s,\theta}},\\
\label{ouv}\|\omega_\l(D)(u\,v)\|_{X^{s,\theta-1}}\lesssim \|u\|_{X^{s,\theta}}\|v\|_{X^{s,\theta}},\\
\label{oubv}\|\omega_\l(D)(\overline{u}\,v)\|_{X^{s,\theta-1}}\lesssim \|u\|_{X^{s,\theta}}\|v\|_{X^{s,\theta}},
\end{align}
where $\l\geq 1$ is an arbitrary scaling parameter, $z_\l=z_\l(x)=\l^{-2}z(\l^{-1}x)$, and the multiplier operator $\omega_\l=\omega_\l(D)$ has the symbol $\omega_\l(\xi)=\omega(\l\xi)$. With the exception of the bilinear estimates, the other ones are by now somewhat classical with \eqref{zhs} and \eqref{wst-1} being directly argued from \eqref{deh-hs} and \eqref{def-xst}, while \eqref{lineq} appeared in a more general setting in Tao's monograph \cite{T-06} (Proposition 2.12). Furthermore, the way in which we combine \eqref{zhs}-\eqref{oubv} to derive Theorem \ref{main-th-2} mirrors closely the path followed by Kishimoto-Tsugawa in \cite{KT10} to prove their respective results. This is why we provide here only an outline of the argument for Theorem \ref{main-th-2} and refer the interested reader to \cite{KT10} for more details.

\begin{proof}[Sketch of proof for Theorem \ref{main-th-2}]
By letting $\l\geq 1$ denote an arbitrary scaling parameter and taking
\begin{equation*}
v_\l(x,t)=\l^{-2}v(\l^{-1}x,\l^{-2} t) \quad\text{and}\quad v_{0\l}(x)=\l^{-2}v_0(\l^{-1} x),
\end{equation*}
it follows that
\begin{equation}
v_\l(t)=\, S(t)v_{0\l}-i\int_0^tS(t-t') H_\l(v_\l(t'), \overline{v_\l}(t'))\,dt',
\label{stvl}
\end{equation}
where 
\begin{equation*}
H_\l(w, \overline{w}):=\l^{-2}\,\frac{\overline{w}-w}{2}\pm\omega_\l(D)\left(\frac{w+\overline{w}}{2}\right)^2.
\end{equation*}
It is clear that $v$ solves \eqref{stv} on the interval $[0,T]$ if and only if $v_\l$ solves  
\eqref{stvl} on $[0,\l^2T]$. The goal is to show that \eqref{stvl} admits a unique local solution on the time interval $[0,1]$ if $\l$ is chosen sufficiently large. 

For this reason, one works with the following modified version of \eqref{stvl}, 
\begin{equation}
v_\l(t)=\, \varphi(t)S(t)v_{0\l}-i\varphi(t)\int_0^tS(t-t') H_\l(v_\l(t'), \overline{v_\l}(t'))\,dt',
\label{stvl-2}
\end{equation}
and proves that it has a unique global-in-time solution. If we denote the right-hand side of this integral equation, with $v_{0\l}$ fixed, by $I_\l=I_\l(v_\l)$, then an application of \eqref{wst-1}-\eqref{oubv} yields
\begin{equation*}
\aligned
\|I_\l(v_\l)\|&_{X^{s,\theta}}\\
&\lesssim \|v_{0\l}\|_{H^s}+\|H_\l(v_\l, \overline{v_\l})\|_{X^{s,\theta-1}}\\
&\lesssim \|v_{0\l}\|_{H^s}+\l^{-2}\left(\|v_\l\|_{X^{s,\theta-1}}+\|\overline{v_\l}\|_{X^{s,\theta-1}}\right)+\left\|\omega_\l(D)\left(v_\l+\overline{v_\l}\right)^2\right\|_{X^{s,\theta-1}}\\
&\lesssim \|v_{0\l}\|_{H^s}+\l^{-2}\|v_\l\|_{X^{s,\theta}}+\|v_\l\|^2_{X^{s,\theta}}.
\endaligned
\end{equation*}
Similarly, one obtains
\begin{equation*}
\aligned
\|I_\l(v_\l)-I_\l(w_\l)\|_{X^{s,\theta}}\lesssim \left(\l^{-2}+\|v_\l\|_{X^{s,\theta}}+\|w_\l\|_{X^{s,\theta}}\right)\|v_\l-w_\l\|_{X^{s,\theta}}.
\endaligned
\end{equation*}
Based on these two estimates, we argue that for $R\sim \|v_{0\l}\|_{H^s}$ the mapping
\[
I_\l:\{\|w\|_{X^{s,\theta}}\leq R\}\to \{\|w\|_{X^{s,\theta}}\leq R\}
\]
is a contraction if we can choose $\l$ large enough and, at the same time, have\footnote{It is precisely the role of the scaling procedure to make the size of $\|v_{0\l}\|_{H^s}$ small enough to be amenable for the contraction argument.} $\|v_{0\l}\|_{H^s}\lesssim 1$. This is feasible by taking $\l\sim  r^{2/(2s-n+4)}$ and using \eqref{zhs}. Moreover, with this choice, we also obtain that the time of existence for solutions to \eqref{stv} satisfies $T\sim \l^{-2}\sim r^{-4/(2s-n+4)}$. 

The uniqueness claim follows by comparable arguments (also relying on \eqref{Sob-X}), for which we point  to the proof of Proposition 4.1 in \cite{KT10}. 
\end{proof}


\section{Bilinear estimates} \label{bil-section}

In this section, we focus our attention on proving \eqref{oubvb}-\eqref{oubv} and, for this purpose, we first revisit related results obtained by Colliander-Delort-Kenig-Staffilani \cite{CDKS-01} (see also earlier work addressing similar issues by Staffilani \cite{S-97}) and Tao \cite{T-01}. The former paper provided a sharp geometric analysis for bilinear bounds of the type
\begin{align}
\label{ubvb-1}\|\overline{u}\,\overline{v}\|_{X^{\sigma,\theta-1}}\lesssim \|u\|_{X^{s,\theta}}\|v\|_{X^{s,\theta}},\\
\label{uv-1}\|u\,v\|_{X^{\sigma,\theta-1}}\lesssim \|u\|_{X^{s,\theta}}\|v\|_{X^{s,\theta}},\\
\label{ubv-1}\|\overline{u}\,v\|_{X^{\sigma,\theta-1}}\lesssim \|u\|_{X^{s,\theta}}\|v\|_{X^{s,\theta}},
\end{align}
on $\R^{2+1}$ and then used them in the context of LWP for Schr\"{o}dinger equations with quadratic nonlinearities. The article by Tao took up the more general issue of multilinear estimates for arbitrary $X^{s,\theta}$ spaces and developed an abstract framework for proving them, which is now referred to in the literature as  the \emph{$[k;Z]$-multiplier norm method}. As an application of this method, the same paper established the bilinear estimate 
\begin{align}
\label{ubv-2}\|\overline{u}\,v\|_{X^{s,-1/2+\epsilon}}\lesssim \|u\|_{X^{s,1/2-\epsilon}}\|v\|_{X^{s,1/2-\epsilon}}
\end{align}
on $\R^{n+1}$ with $1\leq n\leq 3$, $\epsilon>0$, and $\epsilon\lesssim s+1/4\leq 1/4$, and made the claim that similar arguments lead to 
\begin{align}
\label{ubvb-2}\|\overline{u}\,\overline{v}\|_{X^{s,-1/2+\epsilon}}\lesssim \|u\|_{X^{s,1/2-\epsilon}}\|v\|_{X^{s,1/2-\epsilon}},\\
\label{uv-2}\|uv\|_{X^{s,-1/2+\epsilon}}\lesssim \|u\|_{X^{s,1/2-\epsilon}}\|v\|_{X^{s,1/2-\epsilon}},
\end{align}
on $\R^{n+1}$ when either $n=2$ and $s+3/4\gtrsim\epsilon$ or $n=3$ and $s+1/2\gtrsim\epsilon$.

In line with our main goal, we investigate the validity of \eqref{oubvb}-\eqref{oubv} on $\R^{n+1}$ with $n=2$ or $3$ for pairs of indices $(s,\theta)$ satisfying $s<0$ and $\theta >1/2$. Using the trivial observation
\[
\left|\widehat{\omega_\l(D)w}(\tau,\xi)\right|=\frac{\l^2|\xi|^2}{1+\l^2|\xi|^2}\left|\widehat{w}(\tau,\xi)\right|\leq \left|\widehat{w}(\tau,\xi)\right|,
\]
which yields
\[
\|\omega_\l(D)w\|_{X^{\tilde{s},\tilde{\theta}}}\leq \|w\|_{X^{\tilde{s},\tilde{\theta}}}
\]
for an arbitrary pair $(\tilde{s},\tilde{\theta})$, it follows that it is enough to look at
\begin{align}
\label{cucv-xst}\|\overline{u}\,\overline{v}\|_{X^{s,\theta-1}}\lesssim \|u\|_{X^{s,\theta}}\|v\|_{X^{s,\theta}},\\
\label{uv-xst}\|u\,v\|_{X^{s,\theta-1}}\lesssim \|u\|_{X^{s,\theta}}\|v\|_{X^{s,\theta}},\\
\label{cuv-xst}\|\overline{u}\,v\|_{X^{s,\theta-1}}\lesssim \|u\|_{X^{s,\theta}}\|v\|_{X^{s,\theta}},
\end{align}
under the same conditions for $n$, $s$ and $\theta$. 

Even though one can argue that whatever is needed for proving Theorem \ref{th-main} in terms of bilinear estimates is already covered by \eqref{ubvb-1}-\eqref{ubv-1} and \eqref{ubv-2}-\eqref{uv-2}, we choose to provide a stand-alone proof of \eqref{cucv-xst}-\eqref{cuv-xst} for a number of reasons. One is that we have a unitary argument for both $n=2$ and $n=3$. Another is that we are able to prove \eqref{ubvb-1}-\eqref{uv-1} for indices $\sigma$, $s$, and $\theta$ not covered in \cite{CDKS-01}. Finally, our proof suggests that, in principle, the pairs of indices $(s,\theta)$ for which \eqref{oubvb}-\eqref{oubv} hold true coincide with the ones available for the validity of \eqref{cucv-xst}-\eqref{cuv-xst}. Thus, it is very likely that the functional spaces on which we run the contraction argument need to be modified in order for the Sobolev regularity in Theorem \ref{th-main}  to be lowered.

In arguing for \eqref{cucv-xst}-\eqref{cuv-xst}, we rely on Tao's methodology, which is directly specialized to our setting. We denote
\[
\Gamma_3(\R^n\times \R)=\{((\xi_1,\tau_1),(\xi_2,\tau_2),(\xi_3,\tau_3))\in (\R^n\times \R)^3;  (\xi_1,\tau_1)+(\xi_2,\tau_2)+(\xi_3,\tau_3)=0\}
\]
and define
\[
\int_{\Gamma_3(\R^n\times \R)}f:=\int_{(\R^n\times \R)^2} f((\xi_1,\tau_1),(\xi_2,\tau_2),(-\xi_1-\xi_2,-\tau_1-\tau_2))\,d\xi_1d\tau_1d\xi_2d\tau_2.
\]
Any function $m:\Gamma_3(\R^n\times \R)\to\mathbb{C}$ is called a $[3; \R^n\times \R]$-multiplier and we let $\|m\|_{[3; \R^n\times \R]}$ denote the best constant for which
\begin{equation*}
\aligned
\bigg|\int_{\Gamma_3(\R^n\times \R)}m((\xi_1,\tau_1),&(\xi_2,\tau_2),(\xi_3,\tau_3))f_1(\xi_1,\tau_1)f_2(\xi_2,\tau_2)f_3(\xi_3,\tau_3)\bigg|\\
\leq&\|m\|_{[3; \R^n\times \R]}\|f_1\|_{L^2(\R^n\times \R)}\|f_2\|_{L^2(\R^n\times \R)}\|f_3\|_{L^2(\R^n\times \R)} 
\endaligned
\end{equation*}
is valid for all test functions $(f_i)_{1\leq i\leq 3}$ on $\R^n\times \R$. 

If we take for example \eqref{cucv-xst}, then, by applying duality and Plancherel's theorem, we can rewrite it equivalently as 
\begin{equation*}
\aligned
\bigg|\int_{\Gamma_3(\R^n\times \R)}\widehat{\overline{u}}(\xi_1,\tau_1)\widehat{\overline{v}}(\xi_2,&\tau_2)\widehat{\overline{w}}(\xi_3,\tau_3)\bigg|\\
\sim&\,\left|\int_{\R^n\times\R}\overline{u}(x,t)\overline{v}(x,t)\overline{w}(x,t)\,dxdt\right|\\
\lesssim&\,\|u\|_{X^{s,\theta}}\|v\|_{X^{s,\theta}} \|w\|_{X^{-s,1-\theta}}\\
=&\,\|\langle\xi\rangle^s\langle\tau-|\xi|^2\rangle^\theta\widehat{u}(\xi,\tau)\|_{L^2_{\xi,\tau}}\|\langle\xi\rangle^s\langle\tau-|\xi|^2\rangle^\theta\widehat{v}(\xi, \tau)\|_{L^2_{\xi,\tau}}\\
&\cdot\,\|\langle\xi\rangle^{-s}\langle\tau-|\xi|^2\rangle^{1-\theta}\widehat{w}(\xi, \tau)\|_{L^2_{\xi,\tau}}, 
\endaligned
\end{equation*}
which can be easily turned into
\begin{equation*}
\aligned
\bigg|\int_{\Gamma_3(\R^n\times \R)} \frac{\langle\xi_3\rangle^s\langle\tau_3+|\xi_3|^2\rangle^{\theta-1}}{\langle\xi_1\rangle^s\langle\tau_1+|\xi_1|^2\rangle^{\theta}\langle\xi_2\rangle^s\langle\tau_2+|\xi_2|^2\rangle^{\theta}}f_1(\xi_1,\tau_1)&f_2(\xi_2,\tau_2)f_3(\xi_3,\tau_3)\bigg|\\
\lesssim \|f_1\|_{L^2(\R^n\times \R)}\|f_2\|_{L^2(\R^n\times \R)}&\|f_3\|_{L^2(\R^n\times \R)}. 
\endaligned
\end{equation*}
Thus, according to the above definitions, proving \eqref{cucv-xst} is identical to showing that 
\begin{equation}
\left\|\frac{\langle\xi_3\rangle^s\langle\tau_3+|\xi_3|^2\rangle^{\theta-1}}{\langle\xi_1\rangle^s\langle\tau_1+|\xi_1|^2\rangle^{\theta}\langle\xi_2\rangle^s\langle\tau_2+|\xi_2|^2\rangle^{\theta}}\right\|_{[3,\R^n\times \R]}\lesssim 1
\label{cucv-norm}
\end{equation}
holds true, with similar multiplier-norm estimates being available for both \eqref{uv-xst} and \eqref{cuv-xst}. In fact, these new bounds can be stated generically in the form
\begin{equation}
\left\|\frac{\langle\xi_1\rangle^{-s}\langle\xi_2\rangle^{-s}\langle\xi_3\rangle^s}{\langle\tau_1-h_1(\xi_1)\rangle^{\theta}\langle\tau_2-h_2(\xi_2)\rangle^{\theta}\langle\tau_3-h_3(\xi_3)\rangle^{1-\theta}}\right\|_{[3,\R^n\times \R]}\lesssim 1,
\label{gen-mult}
\end{equation}
where $h_i(\xi)=\pm|\xi|^2$ for all $1\leq i\leq 3$.

At this point, Tao introduces the notation
\begin{equation*}
\lambda_i=\tau_i-h_i(\xi_i), \qquad 1\leq i\leq 3,
\end{equation*}
and defines the \emph{resonance function} $h:\Gamma_3(\R^n)\to \R$ by
\begin{equation}
h(\xi_1,\xi_2,\xi_3):=h_1(\xi_1)+h_2(\xi_2)+h_3(\xi_3).
\label{h-def}
\end{equation}
It is easy to see that on the support of the multiplier in \eqref{gen-mult} we have 
\begin{equation}
\xi_1+\xi_2+\xi_3=0\qquad\text{and}\qquad \lambda_1+\lambda_2+\lambda_3+h(\xi_1,\xi_2,\xi_3)=0.
\label{xlh}
\end{equation}
Next, it is argued that one can reduce the proof of \eqref{gen-mult} to the case when 
\begin{equation*}
\min\{|\lambda_1|,|\lambda_2|,|\lambda_3|\}\gtrsim 1 \qquad\text{and}\qquad \max\{|\xi_1|,|\xi_2|,|\xi_3|\}\gtrsim 1.
\end{equation*}
Following this, a dyadic decomposition for $(\xi_i)_{1\leq i\leq 3}$, $(\lambda_i)_{1\leq i\leq 3}$, and $h$ is performed and one infers
\begin{equation*}
\aligned
&\text{(LHS) of \eqref{gen-mult}}\\
&\qquad\lesssim \left\|\sum_{\max N_i\gtrsim 1}\sum_H\sum_{\min L_i\gtrsim 1} \frac{\langle N_1\rangle^{-s}\langle N_2\rangle^{-s}\langle N_3\rangle^{s}}{L_1^\theta L_2^\theta L_3^{1-\theta}}X_{N_1, N_2, N_3; H; L_1, L_2, L_3}\right\|_{[3,\R^n\times \R]}
\endaligned
\end{equation*}
where
\begin{equation}
\aligned
X_{N_1, N_2, N_3; H; L_1, L_2, L_3}&=X_{N_1, N_2, N_3; H; L_1, L_2, L_3}((\xi_1,\tau_1),(\xi_2,\tau_2),(\xi_3,\tau_3))\\
&:=\chi_{|h(\xi_1,\xi_2,\xi_3)|\sim H}\prod_{1\leq i\leq 3}\left(\chi_{|\xi_i|\sim N_i}\chi_{|\lambda_i|\sim L_i}\right)
\endaligned
\label{x-def}
\end{equation}
and $(N_i)_{1\leq i\leq 3}$, $(L_i)_{1\leq i\leq 3}$, and $H\in 2^{\mathbb Z}$. If we let $N_{max}\geq N_{med}\geq N_{min}$ denote the values of $N_1$, $N_2$, and $N_3$ in decreasing order, with a similar notation for the values of $L_1$, $L_2$, and $L_3$, then, based on \eqref{xlh}, we deduce that
\begin{equation}
N_{max}\sim N_{med} \qquad\text{and}\qquad L_{max}\sim \max\{H, L_{med}\}
\label{Nmax-Lmax}
\end{equation}
need to be valid in order for $X_{N_1, N_2, N_3; H; L_1, L_2, L_3}$ not to vanish. 

Using also the relative orthogonality of the dyadic decomposition, Tao is able to derive initially that
\begin{equation*}
\aligned
&\text{(LHS) of \eqref{gen-mult}}\\
&\qquad\lesssim \sup_{N\gtrsim 1}\Bigg\|\sum_{N_{max}\sim N_{med}\sim N}\sum_H\sum_{L_{max}\sim \max\{H, L_{med}\}} \frac{\langle N_1\rangle^{-s}\langle N_2\rangle^{-s}\langle N_3\rangle^{s}}{L_1^\theta L_2^\theta L_3^{1-\theta}}\\
&\quad\qquad\qquad\qquad\qquad\qquad\qquad\qquad\qquad\qquad\qquad\quad\cdot X_{N_1, N_2, N_3; H; L_1, L_2, L_3}\Bigg\|_{[3,\R^n\times \R]}
\endaligned
\end{equation*}
where the summation in the inner and the outer sums is in fact performed over all $L_i$'s and $N_i$'s, respectively, obeying the restriction listed under the sums\footnote{Similar summation conventions are used throughout this section. See also Section 2 in \cite{T-01}.}. Jointly with the triangle inequality, this implies that, for some $N\gtrsim 1$, at least one of the estimates 
\begin{equation*}
\aligned
\text{(LHS) of \eqref{gen-mult}}\lesssim\sum_{N_{max}\sim N_{med}\sim N}\sum_{L_{min}\gtrsim 1}&\frac{\langle N_1\rangle^{-s}\langle N_2\rangle^{-s}\langle N_3\rangle^{s}}{L_1^\theta L_2^\theta L_3^{1-\theta}}\\
&\cdot\left\| X_{N_1, N_2, N_3; L_{max}; L_1, L_2, L_3}\right\|_{[3,\R^n\times \R]}
\endaligned
\end{equation*}
and 
\begin{equation*}
\aligned
\text{(LHS) of \eqref{gen-mult}}\lesssim \sum_{N_{max}\sim N_{med}\sim N}\sum_{L_{max}\sim L_{med}}\sum_{H\ll L_{max}}&\frac{\langle N_1\rangle^{-s}\langle N_2\rangle^{-s}\langle N_3\rangle^{s}}{L_1^\theta L_2^\theta L_3^{1-\theta}}\\
& \cdot\left\| X_{N_1, N_2, N_3; H; L_1, L_2, L_3}\right\|_{[3,\R^n\times \R]}
\endaligned
\end{equation*}
holds true. In this way, \eqref{gen-mult} would follow if one shows that  
\begin{equation}
\aligned
\sum_{N_{max}\sim N_{med}\sim N}\sum_{L_{min}\gtrsim 1}&\frac{\langle N_1\rangle^{-s}\langle N_2\rangle^{-s}\langle N_3\rangle^{s}}{L_1^\theta L_2^\theta L_3^{1-\theta}}\\
&\cdot\left\| X_{N_1, N_2, N_3; L_{max}; L_1, L_2, L_3}\right\|_{[3,\R^n\times \R]}\lesssim 1
\endaligned
\label{lm-gen}
\end{equation}
and 
\begin{equation}
\aligned
\sum_{N_{max}\sim N_{med}\sim N}\sum_{L_{max}\sim L_{med}}\sum_{H\ll L_{max}}&\frac{\langle N_1\rangle^{-s}\langle N_2\rangle^{-s}\langle N_3\rangle^{s}}{L_1^\theta L_2^\theta L_3^{1-\theta}}\\
&\cdot\left\| X_{N_1, N_2, N_3; H; L_1, L_2, L_3}\right\|_{[3,\R^n\times \R]}\lesssim 1,
\endaligned
\label{hm-gen}
\end{equation}
for all values of $N\gtrsim 1$. Tao calls the setting of the first bound (i.e., $H\sim L_{max}$) the \emph{low modulation} case and the one for the second bound (i.e., $L_{max}\sim L_{med}\gg H$) the \emph{high modulation} case.

The first part of the argument for proving \eqref{lm-gen} and \eqref{hm-gen} consists in estimating the two multiplier norms and this has been achieved by Tao in a sharp manner. Given \eqref{h-def}, \eqref{x-def}, and the existing symmetries, the analysis is reduced to two scenarios. The so-called $(+++)$ case happens when $h_1(\xi)=h_2(\xi)=h_3(\xi)=|\xi|^2$ and, hence,
\begin{equation}
H\sim |h|=|\xi_1|^2+|\xi_2|^2+|\xi_3|^2\sim N_{max}^2.
\label{H-Nmax}
\end{equation}
The other instance, named the  $(++-)$ case, takes place when $h_1(\xi)=h_2(\xi)=-h_3(\xi)=|\xi|^2$ and, due to \eqref{xlh}, one has
\begin{equation}
H\sim |h|=\left||\xi_1|^2+|\xi_2|^2-|\xi_3|^2\right|=2|\xi_1\cdot \xi_2|\lesssim N_1N_2.
\label{H-N1-N2}
\end{equation}
The following are the combined outcomes of Propositions 11.1 and 11.2 in \cite{T-01} when $n\geq 2$.

\begin{lemma} Let $n\geq 2$ and take $N_1$, $N_2$, $N_3$, $L_1$, $L_2$, $L_3$, and $H$ to be positive numbers satisfying \eqref{Nmax-Lmax}.

\begin{itemize}

\item \textbf{$(+++)$ case}: both \eqref{H-Nmax} and
\begin{equation}
\left\| X_{N_1, N_2, N_3; H; L_1, L_2, L_3}\right\|_{[3,\R^n\times \R]}\lesssim L_{min}^{\frac 12}N_{max}^{-\frac 12}N_{min}^{\frac{n-1}{2}} \min\{N_{max}N_{min}, L_{med}\}^{\frac 12}
\label{loc+++}
\end{equation}
are valid.

\item \textbf{$(++-)$ case}: \eqref{H-N1-N2} holds true and

\begin{enumerate}

\item if $N_1\sim N_2\gg N_3$, the multiplier norm vanishes unless $H\sim N_1^2$ and, in this situation,
\begin{equation}
\left\| X_{N_1, N_2, N_3; H; L_1, L_2, L_3}\right\|_{[3,\R^n\times \R]}\lesssim L_{min}^{\frac 12}N_{max}^{-\frac 12}N_{min}^{\frac{n-1}{2}} \min\{N_{max}N_{min}, L_{med}\}^{\frac 12}
\label{++-1}
\end{equation}
is valid;

\item if $N_1\sim N_3\gg N_2$ and $H\sim L_2\gg L_1$, $L_3$, $N^2_2$, then
\begin{equation}
\left\| X_{N_1, N_2, N_3; H; L_1, L_2, L_3}\right\|_{[3,\R^n\times \R]}\lesssim L_{min}^{\frac 12}N_{max}^{-\frac 12}N_{min}^{\frac{n-1}{2}} \min\left\{H,\frac{H}{N^2_{min}} L_{med}\right\}^{\frac 12}
\label{++-2}
\end{equation}
is valid. The same estimate holds true if the roles of indices $1$ and $2$ are reversed. This is also called the \textbf{coherence subcase};

\item in all other instances not covered above and for $\epsilon>0$,
\begin{equation}
\aligned
\left\| X_{N_1, N_2, N_3; H; L_1, L_2, L_3}\right\|_{[3,\R^n\times \R]}\lesssim\, &L_{min}^{\frac 12}N_{max}^{-\frac 12}N_{min}^{\frac{n-1}{2}}\\ 
&\cdot\min\left\{H, L_{med}\right\}^{\frac 12}\min\left\{1, \frac{H}{N^2_{min}}\right\}^{\frac 12 -\epsilon}
\label{++-3}
\endaligned
\end{equation}
is valid, with the implicit constant depending on $\epsilon$. If $n=2$, $\epsilon$ can be removed. 

\end{enumerate}  

\end{itemize}
\label{lm-++-}
\end{lemma}

The second part of the proof for \eqref{lm-gen} and \eqref{hm-gen} consists in using the multiplier norm bounds from the previous lemma and performing the two summations. This is where we start, in earnest, our own argument. The following definition describes the indices $s$ and $\theta$ relevant to our analysis.

\begin{definition}
We say that the triplet $(n,s,\theta)$ is admissible if either
\begin{equation}
n=2, \qquad \frac{1}{2}< \theta\neq \frac{3}{4}, \qquad \max\left\{\theta-\frac{5}{4}, 2\theta-2\right\}\leq s<0,
\label{2t1}
\end{equation}
or
\begin{equation}
n=2,\qquad \theta=\frac{3}{4}, \qquad-\frac{1}{2}< s<0,
\label{2t2}
\end{equation}
or 
\begin{equation}
n=3, \qquad \theta>\frac{1}{2}, \qquad 2\theta-\frac{3}{2}\leq s<0.\label{3t}
\end{equation}
\label{nst}
\end{definition}

\begin{remark}
It is easy to verify that if $(n,s,\theta)$ is admissible then
\begin{equation}
s\geq 2\theta +\frac{n-6}{2} >\frac{n-4}{2}.
\label{ns}
\end{equation}
Moreover, if 
\begin{equation}
n=2 \ \text{or} \ n=3, \qquad \theta>1/2, \qquad
\frac{\theta-1}{2}<s<0,
\label{23ts}
\end{equation}
then a direct argument shows that $(n,s,\theta)$ is admissible.
\label{Rem}
\end{remark}


\begin{prop}
The bilinear estimate \eqref{cucv-xst} is valid if $(n,s,\theta)$ is admissible.
\label{cucv-prop}
\end{prop}

\begin{proof}
As argued before, the bound to be proven is equivalent to \eqref{cucv-norm} which, by using the compatible transformation $(\tau_1,\tau_2,\tau_3)\mapsto (-\tau_1,-\tau_2,-\tau_3)$, becomes
\[
\left\|\frac{\langle\xi_3\rangle^s\langle\tau_3-|\xi_3|^2\rangle^{\theta-1}}{\langle\xi_1\rangle^s\langle\tau_1-|\xi_1|^2\rangle^{\theta}\langle\xi_2\rangle^s\langle\tau_2-|\xi_2|^2\rangle^{\theta}}\right\|_{[3,\R^n\times \R]}\lesssim 1.
\] 
We are in the $(+++)$ case and we would be done if we show that \eqref{lm-gen} and \eqref{hm-gen} hold true in this setting. According to \eqref{H-Nmax}, we can assume $H\sim N^2_{max}\sim N^2$ and, since $s<0$ and $\theta>1/2$, we deduce
\begin{equation}
\frac{\langle N_1\rangle^{-s}\langle N_2\rangle^{-s}\langle N_3\rangle^{s}}{L_1^\theta L_2^\theta L_3^{1-\theta}}\lesssim \frac{N^{-2s}\langle N_{min}\rangle^{s}}{L_{min}^\theta L_{med}^\theta L_{max}^{1-\theta}}.
\label{NLst}
\end{equation}

We treat first \eqref{lm-gen}, for which one has $L_{max}\sim H\sim N^2$. If we take advantage jointly of \eqref{loc+++}, \eqref{NLst}, and $\theta>1/2$, then we can estimate the left-hand side of \eqref{lm-gen} by
\begin{equation*}
\aligned
&\text{(LHS) of \eqref{lm-gen}}\\
&\quad\lesssim N^{-2s+2\theta-\frac{5}{2}}\sum_{N_{min}\lesssim N}\ \sum_{1\lesssim L_{min}\leq L_{med}\lesssim N^2} \Big(\langle N_{min}\rangle^{s} N_{min}^{\frac{n-1}{2}}\\
&\quad\qquad\qquad\qquad\qquad\qquad\qquad\qquad\qquad\qquad\cdot L_{min}^{\frac{1}{2}-\theta}L_{med}^{-\theta} \min\{NN_{min}, L_{med}\}^{\frac 12}\Big)\\
&\quad\lesssim N^{-2s+2\theta-2}\sum_{N_{min}\lesssim N^{-1}}\ \sum_{1\lesssim L_{med}\lesssim N^2} N_{min}^{\frac{n}{2}}L_{med}^{-\theta}\\
&\quad \ \ \ + N^{-2s+2\theta-\frac{5}{2}}\sum_{N^{-1}\lesssim N_{min}\lesssim N}\ \sum_{1\lesssim L_{med}\lesssim NN_{min}} \langle N_{min}\rangle^{s} N_{min}^{\frac{n-1}{2}}L_{med}^{\frac{1}{2}-\theta}\\
&\quad \ \ \ + N^{-2s+2\theta-2}\sum_{N^{-1}\lesssim N_{min}\lesssim N}\ \sum_{NN_{min}\lesssim L_{med}\lesssim N^2} \langle N_{min}\rangle^{s} N_{min}^{\frac{n}{2}}L_{med}^{-\theta}\\
&\quad\lesssim N^{-2s+2\theta-2-\frac{n}{2}}+N^{-2s+2\theta-\frac{5}{2}}\sum_{N^{-1}\lesssim N_{min}\lesssim N}\langle N_{min}\rangle^{s} N_{min}^{\frac{n-1}{2}}\left(1+(NN_{min})^{\frac{1}{2}-\theta}\right)\\
&\quad \sim N^{-2s+2\theta-\frac{5}{2}}\left(1+\sum_{1\lesssim N_{min}\lesssim N}N_{min}^{s+\frac{n-1}{2}}\right).
\endaligned
\end{equation*}
A simple analysis based on how $s+(n-1)/2$ compares to $0$ yields that   
\[
N^{-2s+2\theta-\frac{5}{2}}\left(1+\sum_{1\lesssim N_{min}\lesssim N}N_{min}^{s+\frac{n-1}{2}}\right)\lesssim 1
\]  
if and only if  $(n,s,\theta)$ is admissible.

Next, we address \eqref{hm-gen}, for which we work with  $L_{max}\sim L_{med}\gg H\sim N^2$. This implies 
\begin{equation}
NN_{min}\lesssim N^2\ll L_{med},
\label{lmed-n-nmin}
\end{equation}
which leads to 
\begin{equation}
\min\{NN_{min}, L_{med}\}\sim NN_{min}.
\label{lmed-n-nmin-2}
\end{equation}
Together with \eqref{loc+++},\eqref{NLst}, and $\theta>1/2$, this fact allows us to infer
\begin{equation*}
\aligned
&\text{(LHS) of \eqref{hm-gen}}\\
&\qquad\qquad\lesssim N^{-2s}\sum_{N_{min}\lesssim N}\ \sum_{\stackrel{1\lesssim L_{min}\leq L_{med}\sim L_{max}}{N^2\ll L_{max}}} \langle N_{min}\rangle^{s} N_{min}^{\frac{n}{2}}L_{min}^{\frac{1}{2}-\theta}L_{max}^{-1}\\
&\qquad\qquad\lesssim N^{-2s-2}\sum_{N_{min}\lesssim N}\langle N_{min}\rangle^{s}N_{min}^{\frac{n}{2}} \sim N^{-2s-2}\left(1+\sum_{1\lesssim N_{min}\lesssim N}N_{min}^{s+\frac{n}{2}}\right).
\endaligned
\end{equation*}
Using now \eqref{ns}, we deduce
\[
N^{-2s-2}\left(1+\sum_{1\lesssim N_{min}\lesssim N}N_{min}^{s+\frac{n}{2}}\right)\sim N^{-s+\frac{n-4}{2}}\lesssim 1
\]
and the argument is concluded.
\end{proof}


\begin{prop}
The bilinear estimate \eqref{uv-xst} is valid if $(n,s,\theta)$ is admissible.
\label{uv-prop}
\end{prop}

\begin{proof}
Following the blueprint of deriving \eqref{cucv-norm}, we argue first that  \eqref{uv-xst} is equivalent to 
\[
\left\|\frac{\langle\xi_3\rangle^s\langle\tau_3+|\xi_3|^2\rangle^{\theta-1}}{\langle\xi_1\rangle^s\langle\tau_1-|\xi_1|^2\rangle^{\theta}\langle\xi_2\rangle^s\langle\tau_2-|\xi_2|^2\rangle^{\theta}}\right\|_{[3,\R^n\times \R]}\lesssim 1.
\]
Thus, we need to prove that both \eqref{lm-gen} and \eqref{hm-gen} hold true in the $(++-)$ case. We know that we can rely on \eqref{H-N1-N2} and, for each of the bounds, we have to go through all the three subcases covered in Lemma \ref{lm-++-}.

We start with the analysis for \eqref{lm-gen} and consider first the instance when $N_1\sim N_2\gg N_3$, which also forces $H\sim N_1^2$. Then, based on \eqref{++-1}, we see that we can estimate the left-hand side of \eqref{lm-gen} in identical fashion to the way we estimated it in the previous proposition. Hence, we obtain
\begin{equation*}
\aligned
&\text{(LHS) of \eqref{lm-gen}}\\
&\quad\lesssim N^{-2s+2\theta-\frac{5}{2}}\sum_{N_{min}\lesssim N}\ \sum_{1\lesssim L_{min}\leq L_{med}\lesssim N^2} \Big(\langle N_{min}\rangle^{s} N_{min}^{\frac{n-1}{2}}\\
&\quad\qquad\qquad\qquad\qquad\qquad\qquad\qquad\qquad\qquad\cdot L_{min}^{\frac{1}{2}-\theta}L_{med}^{-\theta} \min\{NN_{min}, L_{med}\}^{\frac 12}\Big)\\
&\quad \lesssim N^{-2s+2\theta-\frac{5}{2}}\left(1+\sum_{1\lesssim N_{min}\lesssim N}N_{min}^{s+\frac{n-1}{2}}\right)
\endaligned
\end{equation*}
and, consequently,  \eqref{lm-gen} is valid in this instance if $(n,s,\theta)$ is admissible. 

If we are in the second scenario of Lemma \ref{lm-++-}, by the symmetry of \eqref{lm-gen} in the indices $1$ and $2$, it is enough to work under the assumption that $N_1\sim N_3\gg N_2$ and $H\sim L_2\gg L_1$, $L_3$, $N_2^2$. Using \eqref{++-2}, $1/2<\theta<1$, and \eqref{ns}, we infer
\begin{equation*}
\aligned
&\text{(LHS) of \eqref{lm-gen}}\\
&\qquad\lesssim N^{-\frac{1}{2}}\sum_{N^{-1}\lesssim N_{min}\lesssim N}\ \sum_{\stackrel{1\lesssim L_{min}\leq L_{med}\leq L_{max}}{N_{min}^2\ll L_{max}\lesssim NN_{min}}}\bigg( \langle N_{min}\rangle^{-s} N_{min}^{\frac{n-1}{2}}\\
&\qquad\qquad\qquad\qquad\qquad\qquad\qquad\qquad\qquad\qquad\cdot L_{min}^{\frac{1}{2}-\theta}L_{med}^{\theta-1}L_{max}^{\frac{1}{2}-\theta}\min\left\{1, \frac{L_{med}}{N_{min}^2}\right\}^{\frac 12}\bigg)\\
&\qquad\lesssim N^{-\frac 12}\sum_{N^{-1}\lesssim N_{min}\lesssim 1}\ \sum_{1\lesssim L_{med}\leq L_{max}\lesssim NN_{min}} N_{min}^{\frac{n-1}{2}}L_{med}^{\theta-1}L_{max}^{\frac{1}{2}-\theta}\\
&\qquad \ \ \ + N^{-\frac 12}\sum_{1\lesssim N_{min}\lesssim N}\ \sum_{\stackrel{1\lesssim L_{med}\lesssim N^2_{min}}{N_{min}^2\ll L_{max}\lesssim NN_{min}}} N_{min}^{-s+\frac{n-3}{2}}L_{med}^{\theta-\frac{1}{2}}L_{max}^{\frac{1}{2}-\theta}\\
&\qquad \ \ \ + N^{-\frac 12}\sum_{1\lesssim N_{min}\lesssim N}\ \sum_{\stackrel{N^2_{min}\lesssim L_{med}\leq L_{max}}{N_{min}^2\ll L_{max}\lesssim NN_{min}}} N_{min}^{-s+\frac{n-1}{2}}L_{med}^{\theta-1}L_{max}^{\frac{1}{2}-\theta}\\
&\qquad\lesssim N^{-\frac 12}\left(1+\sum_{1\lesssim N_{min}\lesssim N}N_{min}^{-s+\frac{n-3}{2}}\right)\lesssim N^{-\frac 12}+\sum_{1\lesssim N_{min}\lesssim N}N_{min}^{-s+\frac{n-4}{2}}\sim 1,
\endaligned
\end{equation*}
which proves \eqref{lm-gen} in this scenario.

To finish the argument for \eqref{lm-gen}, we need to consider the third subcase of the $(++-)$ case in Lemma \ref{lm-++-}, which, reduced by symmetry, comes down to either $N_1\sim N_2\sim N_3$ or $N_1\sim N_3\gtrsim N_2$. For each of them, since $H\sim L_{max}$, we have
\begin{equation}
\min\{H, L_{med}\}\sim L_{med}\qquad \text{and}\qquad \min\left\{1, \frac{H}{N_{min}^2}\right\}\sim \min\left\{1, \frac{L_{max}}{N_{min}^2}\right\}.
\label{hlmed}
\end{equation}
Moreover, since $\theta>1/2$, it follows that
\begin{equation}
\frac{1}{L_1^\theta L_2^\theta L_3^{1-\theta}}\leq \frac{1}{L_{min}^\theta L_{med}^\theta L_{max}^{1-\theta}}.
\label{l1l2l3}
\end{equation}
Therefore, when $N_1\sim N_2\sim N_3$, these two facts together with \eqref{++-3} and $\theta>1/2$ allow us to deduce that
\begin{equation*}
\aligned
\text{(LHS) of \eqref{lm-gen}}
&\lesssim N^{-s+\frac{n-4}{2}+2\epsilon} \sum_{1\lesssim L_{min}\leq L_{med}\leq L_{max}\lesssim N^2} L_{min}^{\frac{1}{2}-\theta}L_{med}^{\frac{1}{2}-\theta}L_{max}^{\theta-\frac{1}{2}-\epsilon}\\
&\lesssim N^{-s+\frac{n-4}{2}+2\epsilon} \sum_{1\lesssim L_{max}\lesssim N^2} L_{max}^{\theta-\frac{1}{2}-\epsilon}.
\endaligned
\end{equation*}
By choosing $0<\epsilon<\theta-1/2$, we argue based on \eqref{ns} that
\begin{equation*}
N^{-s+\frac{n-4}{2}+2\epsilon} \sum_{1\lesssim L_{max}\lesssim N^2} L_{max}^{\theta-\frac{1}{2}-\epsilon}\sim N^{-s+2\theta+\frac{n-6}{2}}\lesssim 1,
\end{equation*}
which yields the desired result.

On the other hand, if we have $N_1\sim N_3\gtrsim N_2$, then, on the basis of \eqref{hlmed}, \eqref{l1l2l3}, \eqref{++-3}, $1/2<\theta<1$, and with the same choice for $\epsilon$, we obtain
\begin{equation*}
\aligned
&\text{(LHS) of \eqref{lm-gen}}\\
&\quad\lesssim N^{-\frac{1}{2}}\sum_{N^{-1}\lesssim N_{min}\lesssim N}\ \sum_{\stackrel{1\lesssim L_{min}\leq L_{med}\leq L_{max}}{L_{max}\lesssim NN_{min}}}\bigg( \langle N_{min}\rangle^{-s} N_{min}^{\frac{n-1}{2}}\\
&\quad\qquad\qquad\qquad\qquad\qquad\qquad\qquad\qquad\qquad\cdot L_{min}^{\frac{1}{2}-\theta}L_{med}^{\frac{1}{2}-\theta}L_{max}^{\theta-1}\min\left\{1, \frac{L_{max}}{N_{min}^2}\right\}^{\frac{1}{2}-\epsilon}\bigg)\\
&\quad\lesssim N^{-\frac 12}\sum_{N^{-1}\lesssim N_{min}\lesssim N}\ \sum_{\langle N_{min}\rangle^2\lesssim L_{max}\lesssim NN_{min}} \langle N_{min}\rangle^{-s}N_{min}^{\frac{n-1}{2}}L_{max}^{\theta-1}\\
&\quad \ \ \ + N^{-\frac 12}\sum_{1\lesssim N_{min}\lesssim N}\ \sum_{1\lesssim L_{max}\lesssim N^2_{min}} N_{min}^{-s+\frac{n-3}{2}+2\epsilon}L_{max}^{\theta-\frac{1}{2}-\epsilon}\\
&\quad\lesssim N^{-\frac 12}\left(\sum_{N^{-1}\lesssim N_{min}\lesssim N}\langle N_{min}\rangle^{-s+2\theta-2}N_{min}^{\frac{n-1}{2}}+\sum_{1\lesssim N_{min}\lesssim N}N_{min}^{-s+2\theta+\frac{n-5}{2}}\right)\\
&\quad\sim N^{-\frac 12}\left(1+\sum_{1\lesssim N_{min}\lesssim N}N_{min}^{-s+2\theta+\frac{n-5}{2}}\right).
\endaligned
\end{equation*}
It can be checked easily that if $(n,s,\theta)$ is admissible, then $-s+2\theta+(n-5)/2\neq 0$. Thus, we derive
\begin{equation}
N^{-\frac 12}\left(1+\sum_{1\lesssim N_{min}\lesssim N}N_{min}^{-s+2\theta+\frac{n-5}{2}}\right)\lesssim N^{-\frac 12}+N^{-s+2\theta+\frac{n-6}{2}}\lesssim 1,
\label{nst-52}
\end{equation}
where the last bound follows according to \eqref{ns}. This finishes the proof of \eqref{lm-gen}.

Next, we address \eqref{hm-gen}, for which the scenario $N_1\sim N_2\gg N_3$ and $H\sim N_1^2$ implies \eqref{lmed-n-nmin} and, hence, \eqref{lmed-n-nmin-2}. Then, we can estimate the left-hand side of \eqref{hm-gen} in exactly the same way as we estimated it in the previous proposition. Thus, we infer
\begin{equation*}
\aligned
\text{(LHS) of \eqref{hm-gen}}&\lesssim N^{-2s}\sum_{N_{min}\lesssim N}\ \sum_{\stackrel{1\lesssim L_{min}\leq L_{med}\sim L_{max}}{N^2\ll L_{max}}} \langle N_{min}\rangle^{s} N_{min}^{\frac{n}{2}}L_{min}^{\frac{1}{2}-\theta}L_{max}^{-1}\\
&\lesssim N^{-2s-2}\left(1+\sum_{1\lesssim N_{min}\lesssim N}N_{min}^{s+\frac{n}{2}}\right)\sim N^{-s+\frac{n-4}{2}}\lesssim 1.
\endaligned
\end{equation*}

The second subcase of the $(++-)$ case in Lemma \ref{lm-++-} does not apply here because $H\ll L_{max}$. The last one can be reduced by symmetry to the instances when either $N_1\sim N_2\sim N_3$ or $N_1\sim N_3\gtrsim N_2$. For each of them, we have
\begin{equation}
\min\{H, L_{med}\}\sim H, \label{h-lmed}
\end{equation}
while for the former we can also rely on
\begin{equation}
\min\left\{1,\frac{H}{N^2_{min}}\right\}\sim \frac{H}{N^2_{min}},
\label{h-nmin}
\end{equation}
due to \eqref{H-N1-N2}. Thus, when $N_1\sim N_2\sim N_3$, we argue based on \eqref{++-3}, applicable to $0<\epsilon<\theta-1/2$, and \eqref{ns} that 
\begin{equation*}
\aligned
&\text{(LHS) of \eqref{hm-gen}}\\
&\qquad\qquad\lesssim N^{-s+\frac{n-4}{2}+2\epsilon}\sum_{1\lesssim L_{min}\leq L_{med}\sim L_{max}} \ \sum_{H\lesssim \min\{L_{max}, N^2\}}L_{min}^{\frac{1}{2}-\theta}L_{max}^{-1}H^{1-\epsilon}\\
&\qquad\qquad\lesssim N^{-s+\frac{n-4}{2}+2\epsilon}\sum_{H\lesssim N^2}\ \sum_{\langle H\rangle\lesssim L_{max}}L_{max}^{-1}H^{1-\epsilon}\\
&\qquad\qquad\lesssim N^{-s+\frac{n-4}{2}+2\epsilon}\sum_{H\lesssim N^2}\langle H\rangle^{-1}H^{1-\epsilon}\sim N^{-s+\frac{n-4}{2}+2\epsilon}\lesssim N^{-s+2\theta+\frac{n-6}{2}}\lesssim 1.
\endaligned
\end{equation*}

For the case when $N_1\sim N_3\gtrsim N_2$, we use again \eqref{++-3} with $0<\epsilon<\theta-1/2$ and \eqref{nst-52} to deduce
\begin{equation*}
\aligned
&\text{(LHS) of \eqref{hm-gen}}\\
&\quad\lesssim N^{-\frac{1}{2}}\sum_{N_{min}\lesssim N}\ \sum_{1\lesssim L_{min}\leq L_{med}\sim L_{max}} \ \sum_{H\lesssim \min\{L_{max}, NN_{min}\}}\bigg(\langle N_{min}\rangle^{-s} N_{min}^{\frac{n-1}{2}}L_{min}^{\frac{1}{2}-\theta}\\
&\quad\qquad\qquad\qquad\qquad\qquad\qquad\qquad\qquad\qquad\qquad\quad \cdot L_{max}^{-1}H^{\frac 12}\min\left\{1,\frac{H}{N^2_{min}}\right\}^{\frac{1}{2}-\epsilon}\bigg)\\
&\quad\lesssim N^{-\frac{1}{2}}\sum_{N_{min}\lesssim N}\, \sum_{H\lesssim NN_{min}} \, \sum_{\langle H \rangle\lesssim L_{max}} \langle N_{min}\rangle^{-s} N_{min}^{\frac{n-1}{2}}L_{max}^{-1}H^{\frac 12}\min\left\{1,\frac{H}{N^2_{min}}\right\}^{\frac{1}{2}-\epsilon}\\
\endaligned
\end{equation*}
\begin{equation*}
\aligned
&\quad\lesssim N^{-\frac{1}{2}}\sum_{N_{min}\lesssim N}\, \sum_{H\lesssim NN_{min}}  \langle N_{min}\rangle^{-s} N_{min}^{\frac{n-1}{2}}\langle H \rangle^{-1}H^{\frac 12}\min\left\{1,\frac{H}{N^2_{min}}\right\}^{\frac{1}{2}-\epsilon}\\
&\quad\lesssim N^{-\frac{1}{2}}\sum_{N_{min}\lesssim N^{-1}}N_{min}^{\frac{n-1}{2}}\bigg(\sum_{H\lesssim N^2_{min}} \frac{H^{1-\epsilon}}{N^{1-2\epsilon}_{min}}+\sum_{N^2_{min}\lesssim H\lesssim NN_{min}}  H^{\frac 12}\bigg)\\
&\quad\quad+ N^{-\frac{1}{2}}\sum_{N^{-1}\lesssim N_{min}\lesssim 1}N_{min}^{\frac{n-1}{2}}\bigg(\sum_{H\lesssim N^2_{min}} \frac{H^{1-\epsilon}}{N^{1-2\epsilon}_{min}}+\sum_{N^2_{min}\lesssim H\lesssim 1}  H^{\frac 12}\\
&\quad\qquad\qquad\qquad\qquad\qquad\qquad\quad+\sum_{1\lesssim H\lesssim NN_{min}}  H^{-\frac 12}\bigg)\\
&\quad\quad+ N^{-\frac{1}{2}}\sum_{1\lesssim N_{min}\lesssim N}N_{min}^{-s+\frac{n-1}{2}}\bigg(\sum_{H\lesssim 1} \frac{H^{1-\epsilon}}{N^{1-2\epsilon}_{min}}+\sum_{1\lesssim H\lesssim N^2_{min}}\frac{H^{-\epsilon}}{N^{1-2\epsilon}_{min}}  \\
&\quad\qquad\qquad\qquad\qquad\qquad\qquad\quad+\sum_{N^2_{min}\lesssim H\lesssim NN_{min}}  H^{-\frac 12}\bigg)\\
&\quad\lesssim \sum_{N_{min}\lesssim N^{-1}}N_{min}^{\frac{n}{2}}+N^{-\frac{1}{2}}\left(\sum_{N^{-1}\lesssim N_{min}\lesssim 1}N_{min}^{\frac{n-1}{2}}+ \sum_{1\lesssim N_{min}\lesssim N}N_{min}^{-s+\frac{n-3}{2}+2\epsilon}\right)\\
&\quad\lesssim N^{-\frac{1}{2}}\left(1+ \sum_{1\lesssim N_{min}\lesssim N}N_{min}^{-s+2\theta+\frac{n-5}{2}}\right)\lesssim 1.
\endaligned
\end{equation*}
This finishes the proof of this proposition. 
\end{proof}


\begin{remark}
Following up on our rationale to argue for \eqref{cucv-xst}-\eqref{cuv-xst}, by comparison to what is proved in \cite{CDKS-01} for \eqref{ubvb-1}-\eqref{uv-1}, one can see that Propositions \ref{cucv-prop} and \ref{uv-prop} cover the previously unknown case for which
\begin{equation*}
\frac{1}{2}<\theta\neq \frac{3}{4}\quad\text{and}\quad\sigma=s=\max\left\{\theta-\frac{5}{4}, 2\theta-2\right\}<0.
\end{equation*}
 
\end{remark}


\begin{prop}
The bilinear estimate \eqref{cuv-xst} is valid if $(n,s,\theta)$ satisfy \eqref{23ts}.
\label{cuv-prop}
\end{prop}

\begin{proof}
As in the case of the previous two results, one recognizes first that the above claim is equivalent to the multiplier norm bound
\[
\left\|\frac{\langle\xi_3\rangle^s\langle\tau_3+|\xi_3|^2\rangle^{\theta-1}}{\langle\xi_1\rangle^s\langle\tau_1+|\xi_1|^2\rangle^{\theta}\langle\xi_2\rangle^s\langle\tau_2-|\xi_2|^2\rangle^{\theta}}\right\|_{[3,\R^n\times \R]}\lesssim 1. 
\]
By using the compatible transformation $(\tau_1,\tau_2,\tau_3)\mapsto (-\tau_1,-\tau_2,-\tau_3)$ and relabeling the indices according to $(1,2,3)\mapsto (1,3,2)$, this can be rewritten as 
\begin{equation}
\left\|\frac{\langle\xi_3\rangle^{-s}\langle\tau_3+|\xi_3|^2\rangle^{-\theta}}{\langle\xi_1\rangle^s\langle\tau_1-|\xi_1|^2\rangle^{\theta}\langle\xi_2\rangle^{-s}\langle\tau_2-|\xi_2|^2\rangle^{1-\theta}}\right\|_{[3,\R^n\times \R]}\lesssim 1.
\label{cuv-norm}
\end{equation}
As in the derivation of \eqref{lm-gen} and \eqref{hm-gen}, the previous estimate would follow if we show that
\begin{equation}
\aligned
\sum_{N_{max}\sim N_{med}\sim N}\sum_{L_1, L_2, L_3\gtrsim 1}&\frac{\langle N_1\rangle^{-s}\langle N_2\rangle^{s}\langle N_3\rangle^{-s}}{L_1^\theta L_2^{1-\theta} L_3^{\theta}}\\
&\quad \cdot\left\| X_{N_1, N_2, N_3; L_{max}; L_1, L_2, L_3}\right\|_{[3,\R^n\times \R]}\lesssim 1
\endaligned
\label{lm++-2}
\end{equation}
and 
\begin{equation}
\aligned
\sum_{N_{max}\sim N_{med}\sim N}\sum_{L_{max}\sim L_{med}}\sum_{H\ll L_{max}}&\frac{\langle N_1\rangle^{-s}\langle N_2\rangle^{s}\langle N_3\rangle^{-s}}{L_1^\theta L_2^{1-\theta} L_3^{\theta}}\\
&\quad \cdot\left\| X_{N_1, N_2, N_3; H; L_1, L_2, L_3}\right\|_{[3,\R^n\times \R]}\lesssim 1
\endaligned
\label{hm++-2}
\end{equation}
hold true for any $N\gtrsim 1$. 

From \eqref{cuv-norm}, we see that we operate in the $(++-)$ case and, as such, we can rely on \eqref{H-N1-N2} and we perform an analysis based on the subcases described in Lemma \ref{lm-++-}. Furthermore, due to \eqref{23ts} and Remark \ref{Rem}, we can also take advantage of \eqref{ns}.

For the low modulation estimate \eqref{lm++-2}, if we are in the $N_1\sim N_2\gg N_3$ scenario, we also have that $H\sim L_{max}\sim N_1^2$. Thus, based on \eqref{++-1}, $1/2<\theta<1$, $s<0$, and \eqref{ns}, we infer
\begin{equation*}
\aligned
&\text{(LHS) of \eqref{lm++-2}}\\
&\quad\lesssim N^{2\theta-\frac{5}{2}}\sum_{N_{min}\lesssim N}\ \sum_{1\lesssim L_{min}\leq L_{med}\lesssim N^2} \Big(\langle N_{min}\rangle^{-s} N_{min}^{\frac{n-1}{2}}\\
&\quad\qquad\qquad\qquad\qquad\qquad\qquad\qquad\qquad\cdot L_{min}^{\frac{1}{2}-\theta}L_{med}^{-\theta} \min\{NN_{min}, L_{med}\}^{\frac 12}\Big)\\
&\quad\lesssim N^{2\theta-2}\sum_{N_{min}\lesssim N^{-1}}\ \sum_{1\lesssim L_{med}\lesssim N^2} N_{min}^{\frac{n}{2}}L_{med}^{-\theta}\\
&\quad \ \ \ + N^{2\theta-\frac{5}{2}}\sum_{N^{-1}\lesssim N_{min}\lesssim N}\ \sum_{1\lesssim L_{med}\lesssim NN_{min}} \langle N_{min}\rangle^{-s} N_{min}^{\frac{n-1}{2}}L_{med}^{\frac{1}{2}-\theta}\\
&\quad \ \ \ + N^{2\theta-2}\sum_{N^{-1}\lesssim N_{min}\lesssim N}\ \sum_{NN_{min}\lesssim L_{med}\lesssim N^2} \langle N_{min}\rangle^{-s} N_{min}^{\frac{n}{2}}L_{med}^{-\theta}\\
&\quad\lesssim N^{2\theta-\frac{n+4}{2}}+N^{2\theta-\frac{5}{2}}\sum_{N^{-1}\lesssim N_{min}\lesssim N}\langle N_{min}\rangle^{-s} N_{min}^{\frac{n-1}{2}}\left(1+(NN_{min})^{\frac{1}{2}-\theta}\right)\\
&\quad \sim N^{2\theta-\frac{n+4}{2}}+N^{-s+2\theta+\frac{n-6}{2}}\lesssim 1.
\endaligned
\end{equation*}

Next, if $N_1\sim N_3\gg N_2$ and $H\sim L_2\gg L_1$, $L_3$, $N_2^2$, then, using \eqref{++-2} and $\theta>1/2$, we derive that
\begin{equation*}
\aligned
&\text{(LHS) of \eqref{lm++-2}}\\
&\qquad\lesssim N^{-2s-\frac{1}{2}}\sum_{N^{-1}\lesssim N_{min}\lesssim N}\ \sum_{\stackrel{1\lesssim L_{min}\leq L_{med}\leq L_{max}}{N_{min}^2\ll L_{max}\lesssim NN_{min}}}\bigg( \langle N_{min}\rangle^{s} N_{min}^{\frac{n-1}{2}}\\
&\qquad\qquad\qquad\qquad\qquad\qquad\qquad\qquad\qquad\qquad\cdot L_{min}^{\frac{1}{2}-\theta}L_{med}^{-\theta}L_{max}^{\theta-\frac{1}{2}}\min\left\{1, \frac{L_{med}}{N_{min}^2}\right\}^{\frac 12}\bigg)\\
&\qquad\lesssim N^{-2s-\frac 12}\sum_{N^{-1}\lesssim N_{min}\lesssim 1}\ \sum_{1\lesssim L_{med}\leq L_{max}\lesssim NN_{min}} N_{min}^{\frac{n-1}{2}}L_{med}^{-\theta}L_{max}^{\theta-\frac{1}{2}}\\
&\qquad \ \ \ + N^{-2s-\frac 12}\sum_{1\lesssim N_{min}\lesssim N}\ \sum_{\stackrel{1\lesssim L_{med}\lesssim N^2_{min}}{N_{min}^2\ll L_{max}\lesssim NN_{min}}} N_{min}^{s+\frac{n-3}{2}}L_{med}^{\frac{1}{2}-\theta}L_{max}^{\theta-\frac{1}{2}}\\
\endaligned
\end{equation*}
\begin{equation*}
\aligned
&\qquad \ \ \ + N^{-2s-\frac 12}\sum_{1\lesssim N_{min}\lesssim N}\ \sum_{\stackrel{N^2_{min}\lesssim L_{med}\leq L_{max}}{N_{min}^2\ll L_{max}\lesssim NN_{min}}} N_{min}^{s+\frac{n-1}{2}}L_{med}^{-\theta}L_{max}^{\theta-\frac{1}{2}}\\
&\qquad\lesssim N^{-2s+\theta-1}\left(1+\sum_{1\lesssim N_{min}\lesssim N}N_{min}^{s+\theta+\frac{n-4}{2}}\right).
\endaligned
\end{equation*}
When $n=2$, we argue that $s<0$ and $\theta<1$ imply
\[
s+\theta+\frac{n-4}{2}<0
\]
and, thus, \eqref{lm++-2} is valid if $s\geq(\theta-1)/2$. When $n=3$ and $(n,s,\theta)$ is admissible, it is easy to check that $s+\theta+(n-4)/2$ can be either negative, positive, or equal to zero. If it is negative, then, as above, $s\geq(\theta-1)/2$ is a sufficient condition for  \eqref{lm++-2} to hold true. If it is positive, then we deduce with the help of \eqref{ns} that 
\begin{equation*}
\aligned
\text{(LHS) of \eqref{lm++-2}}\lesssim N^{-2s+\theta-1}+N^{-s+2\theta+\frac{n-6}{2}}\lesssim N^{-2s+\theta-1}+1
\endaligned
\end{equation*}
and, yet again, \eqref{lm++-2} is valid if $s\geq(\theta-1)/2$. If $s+\theta+(n-4)/2=0$, then we infer that
\begin{equation*}
\aligned
\text{(LHS) of \eqref{lm++-2}}\lesssim N^{-2s+\theta-1}\ln N\endaligned
\end{equation*}
and we need to impose the stricter condition $s>(\theta-1)/2$ for  \eqref{lm++-2} to hold true.

Given that, unlike \eqref{lm-gen}, \eqref{lm++-2} is not symmetric in the indices $1$ and $2$, we also need to consider the scenario when $N_2\sim N_3\gg N_1$ and $H\sim L_1\gg L_2$, $L_3$, $N_1^2$. In this situation, an application of \eqref{++-2} yields
\begin{equation*}
\aligned
&\text{(LHS) of \eqref{lm++-2}}\\
&\qquad\lesssim N^{-\frac{1}{2}}\sum_{N^{-1}\lesssim N_{min}\lesssim N}\ \sum_{\stackrel{1\lesssim L_{min}\leq L_{med}\leq L_{max}}{N_{min}^2\ll L_{max}\lesssim NN_{min}}}\bigg( \langle N_{min}\rangle^{-s} N_{min}^{\frac{n-1}{2}}\\
&\qquad\qquad\qquad\qquad\qquad\qquad\qquad\qquad\qquad\qquad\cdot L_{min}^{\frac{1}{2}-\theta}L_{med}^{\theta-1}L_{max}^{\frac{1}{2}-\theta}\min\left\{1, \frac{L_{med}}{N_{min}^2}\right\}^{\frac 12}\bigg),
\endaligned
\end{equation*}
which is identical with the estimate satisfied by the left-hand side of \eqref{lm-gen} for the subcase when $N_1\sim N_3\gg N_2$ and $H\sim L_2\gg L_1$, $L_3$, $N_2^2$. Hence, 
\begin{equation*}
\aligned
\text{(LHS) of \eqref{lm++-2}}\lesssim N^{-\frac 12}+\sum_{1\lesssim N_{min}\lesssim N}N_{min}^{-s+\frac{n-4}{2}}\sim 1.
\endaligned
\end{equation*}

In order to conclude the proof of \eqref{lm++-2}, we need to investigate the third subcase, which can be reduced to $N_1\sim N_2\sim N_3$, $N_2\sim N_3\gtrsim N_1$, and $N_1\sim N_3\gtrsim N_2$, without making extra assumptions. As in the previous proposition, in addition to $L_{max}\sim H\lesssim N_1N_2$, we can rely on \eqref{hlmed} and, since $\theta>1/2$, on
\begin{equation}
\frac{1}{L_1^\theta L_2^{1-\theta} L_3^{\theta}}\leq \frac{1}{L_{min}^\theta L_{med}^\theta L_{max}^{1-\theta}}
\label{l1l2l3-v2}
\end{equation}
for either of these scenarios. 

If $N_1\sim N_2\sim N_3$, then \eqref{++-3} implies
\begin{equation*}
\aligned
\text{(LHS) of \eqref{lm++-2}}
\lesssim N^{-s+\frac{n-4}{2}+2\epsilon} \sum_{1\lesssim L_{min}\leq L_{med}\leq L_{max}\lesssim N^2} L_{min}^{\frac{1}{2}-\theta}L_{med}^{\frac{1}{2}-\theta}L_{max}^{\theta-\frac{1}{2}-\epsilon},
\endaligned
\end{equation*}
which coincides with the initial bound satisfied by the left-hand side of \eqref{lm-gen} in the same situation. Thus, with the appropriate choice for $\epsilon$ (i.e., $0<\epsilon<\theta-1/2$), we obtain
\begin{equation*}
\aligned
\text{(LHS) of \eqref{lm++-2}}
\lesssim N^{-s+2\theta+\frac{n-6}{2}}\lesssim 1.
\endaligned
\end{equation*}

When $N_2\sim N_3\gtrsim N_1$, we use \eqref{hlmed}, \eqref{l1l2l3-v2}, and \eqref{++-3} to derive that
\begin{equation*}
\aligned
\text{(LHS) of \eqref{lm++-2}}\lesssim N^{-\frac{1}{2}}\sum_{N^{-1}\lesssim N_{min}\lesssim N}\ &\sum_{\stackrel{1\lesssim L_{min}\leq L_{med}\leq L_{max}}{L_{max}\lesssim NN_{min}}}\bigg( \langle N_{min}\rangle^{-s} N_{min}^{\frac{n-1}{2}}\\ 
&\cdot L_{min}^{\frac{1}{2}-\theta}L_{med}^{\frac{1}{2}-\theta}L_{max}^{\theta-1}\min\left\{1, \frac{L_{max}}{N_{min}^2}\right\}^{\frac{1}{2}-\epsilon}\bigg).
\endaligned
\end{equation*}
This estimate is identical to the one satisfied by the left-hand side of \eqref{lm-gen} when $N_1\sim N_3\gtrsim N_2$ and, thus, \eqref{lm++-2} holds true if $(n,s,\theta)$ is admissible.

If $N_1\sim N_3\gtrsim N_2$, then we can apply \eqref{hlmed}, \eqref{l1l2l3-v2}, \eqref{++-3}, and $1/2<\theta<1$, and take $0<\epsilon<\theta-1/2$ to argue that
\begin{equation*}
\aligned
&\text{(LHS) of \eqref{lm++-2}}\\
&\quad\lesssim N^{-2s-\frac{1}{2}}\sum_{N^{-1}\lesssim N_{min}\lesssim N}\ \sum_{\stackrel{1\lesssim L_{min}\leq L_{med}\leq L_{max}}{L_{max}\lesssim NN_{min}}}\bigg( \langle N_{min}\rangle^{s} N_{min}^{\frac{n-1}{2}}\\
&\quad\qquad\qquad\qquad\qquad\qquad\qquad\qquad\qquad\quad\cdot L_{min}^{\frac{1}{2}-\theta}L_{med}^{\frac{1}{2}-\theta}L_{max}^{\theta-1}\min\left\{1, \frac{L_{max}}{N_{min}^2}\right\}^{\frac{1}{2}-\epsilon}\bigg)\\
&\quad\lesssim N^{-2s-\frac 12}\sum_{N^{-1}\lesssim N_{min}\lesssim N}\ \sum_{\langle N_{min}\rangle^2\lesssim L_{max}\lesssim NN_{min}} \langle N_{min}\rangle^{s}N_{min}^{\frac{n-1}{2}}L_{max}^{\theta-1}\\
&\quad \ \ \ + N^{-2s-\frac 12}\sum_{1\lesssim N_{min}\lesssim N}\ \sum_{1\lesssim L_{max}\lesssim N^2_{min}} N_{min}^{s+\frac{n-3}{2}+2\epsilon}L_{max}^{\theta-\frac{1}{2}-\epsilon}\\
&\quad\lesssim N^{-2s-\frac 12}\left(\sum_{N^{-1}\lesssim N_{min}\lesssim N}\langle N_{min}\rangle^{s+2\theta-2}N_{min}^{\frac{n-1}{2}}+\sum_{1\lesssim N_{min}\lesssim N}N_{min}^{s+2\theta+\frac{n-5}{2}}\right)\\
&\quad\sim N^{-2s-\frac 12}\left(1+\sum_{1\lesssim N_{min}\lesssim N}N_{min}^{s+2\theta+\frac{n-5}{2}}\right).
\endaligned
\end{equation*}
It is easy to verify that, when $(n,s,\theta)$ is admissible, $s+2\theta+(n-5)/2$ can be either positive, negative, or equal to zero. As such
\begin{equation*}
N^{-2s-\frac 12}\left(1+\sum_{1\lesssim N_{min}\lesssim N}N_{min}^{s+2\theta+\frac{n-5}{2}}\right)\sim N^{-s+2\theta+\frac{n-6}{2}}, N^{-2s-\frac 12},\, \text{or}\ N^{-2s-\frac 12}\ln N,
\end{equation*}
respectively. Due to \eqref{ns}, we see that \eqref{lm++-2} would be valid in this case if we ask for $s>-1/4$, which is a weaker condition than $s>(\theta-1)/2$ imposed before. With this, the argument for 
\eqref{lm++-2} is finished.

Next, we turn to the proof of \eqref{hm++-2}, which is quite similar to the one for \eqref{hm-gen}. If $N_1\sim N_2\gg N_3$ and, hence, $H\sim N_1^2$, then we can rely on \eqref{lmed-n-nmin-2}. Jointly with \eqref{l1l2l3-v2}, \eqref{++-1}, $\theta>1/2$, and \eqref{ns}, it yields 
\begin{equation*}
\aligned
\text{(LHS) of \eqref{hm++-2}}&\lesssim \sum_{N_{min}\lesssim N}\ \sum_{\stackrel{1\lesssim L_{min}\leq L_{med}\sim L_{max}}{N^2\ll L_{max}}} \langle N_{min}\rangle^{-s} N_{min}^{\frac{n}{2}}L_{min}^{\frac{1}{2}-\theta}L_{max}^{-1}\\
&\lesssim N^{-2}\left(1+\sum_{1\lesssim N_{min}\lesssim N}N_{min}^{-s+\frac{n}{2}}\right)\sim N^{-s+\frac{n-4}{2}}\lesssim 1.
\endaligned
\end{equation*}

We have no coherence case to explore since $H\ll L_{max}$. Thus, all we are left to analyze is the stand-alone scenarios $N_1\sim N_2\sim N_3$, $N_2\sim N_3\gtrsim N_1$, and $N_1\sim N_3\gtrsim N_2$. First, we note that we can use \eqref{h-lmed} in all three of these cases. When $N_1\sim N_2\sim N_3$, \eqref{h-nmin} is also available. If we bring \eqref{l1l2l3-v2} and \eqref{++-3} into the mix, then we deduce
\begin{equation*}
\aligned
&\text{(LHS) of \eqref{hm++-2}}\\
&\qquad\qquad\lesssim N^{-s+\frac{n-4}{2}+2\epsilon}\sum_{1\lesssim L_{min}\leq L_{med}\sim L_{max}} \ \sum_{H\lesssim \min\{L_{max}, N^2\}}L_{min}^{\frac{1}{2}-\theta}L_{max}^{-1}H^{1-\epsilon},
\endaligned
\end{equation*}
which coincides with the estimate satisfied by the left-hand side of \eqref{hm-gen} in the same situation. Accordingly, by choosing $0<\epsilon<\theta-1/2$ and applying \eqref{ns}, we infer that \eqref{hm++-2} holds true in this instance.

If $N_2\sim N_3\gtrsim N_1$, then, with the help of \eqref{l1l2l3-v2}, \eqref{++-3}, and \eqref{h-lmed}, we obtain
\begin{equation*}
\aligned
&\text{(LHS) of \eqref{hm++-2}}\\
&\quad\lesssim N^{-\frac{1}{2}}\sum_{N_{min}\lesssim N}\ \sum_{1\lesssim L_{min}\leq L_{med}\sim L_{max}} \ \sum_{H\lesssim \min\{L_{max}, NN_{min}\}}\bigg(\langle N_{min}\rangle^{-s} N_{min}^{\frac{n-1}{2}}L_{min}^{\frac{1}{2}-\theta}\\
&\quad\qquad\qquad\qquad\qquad\qquad\qquad\qquad\qquad\qquad\qquad\quad \cdot L_{max}^{-1}H^{\frac 12}\min\left\{1,\frac{H}{N^2_{min}}\right\}^{\frac{1}{2}-\epsilon}\bigg).
\endaligned
\end{equation*}
This is identical to the bound satisfied by the left-hand side of \eqref{hm-gen} when $N_1\sim N_3\gtrsim N_2$ and, hence, \eqref{hm++-2} is seen to be valid by taking $\epsilon$ as above and relying on \eqref{nst-52}.

When $N_1\sim N_3\gtrsim N_2$, a very similar argument leads to
\begin{equation*}
\aligned
&\text{(LHS) of \eqref{hm++-2}}\\
&\quad\lesssim N^{-2s-\frac{1}{2}}\sum_{N_{min}\lesssim N}\ \sum_{1\lesssim L_{min}\leq L_{med}\sim L_{max}} \ \sum_{H\lesssim \min\{L_{max}, NN_{min}\}}\bigg(\langle N_{min}\rangle^{s} N_{min}^{\frac{n-1}{2}}L_{min}^{\frac{1}{2}-\theta}\\
&\quad\qquad\qquad\qquad\qquad\qquad\qquad\qquad\qquad\qquad\qquad\quad \cdot L_{max}^{-1}H^{\frac 12}\min\left\{1,\frac{H}{N^2_{min}}\right\}^{\frac{1}{2}-\epsilon}\bigg)\\
&\quad\lesssim N^{-2s-\frac{1}{2}}\sum_{N_{min}\lesssim N}\, \sum_{H\lesssim NN_{min}} \, \sum_{\langle H \rangle\lesssim L_{max}} \langle N_{min}\rangle^{s} N_{min}^{\frac{n-1}{2}}L_{max}^{-1}H^{\frac 12}\min\left\{1,\frac{H}{N^2_{min}}\right\}^{\frac{1}{2}-\epsilon}\\
&\quad\lesssim N^{-2s-\frac{1}{2}}\sum_{N_{min}\lesssim N}\, \sum_{H\lesssim NN_{min}}  \langle N_{min}\rangle^{s} N_{min}^{\frac{n-1}{2}}\langle H \rangle^{-1}H^{\frac 12}\min\left\{1,\frac{H}{N^2_{min}}\right\}^{\frac{1}{2}-\epsilon}\\
\endaligned
\end{equation*}
\begin{equation*}
\aligned
&\quad\lesssim N^{-2s-\frac{1}{2}}\sum_{N_{min}\lesssim N^{-1}}N_{min}^{\frac{n-1}{2}}\bigg(\sum_{H\lesssim N^2_{min}} \frac{H^{1-\epsilon}}{N^{1-2\epsilon}_{min}}+\sum_{N^2_{min}\lesssim H\lesssim NN_{min}}  H^{\frac 12}\bigg)\\
&\quad\quad+ N^{-2s-\frac{1}{2}}\sum_{N^{-1}\lesssim N_{min}\lesssim 1}N_{min}^{\frac{n-1}{2}}\bigg(\sum_{H\lesssim N^2_{min}} \frac{H^{1-\epsilon}}{N^{1-2\epsilon}_{min}}+\sum_{N^2_{min}\lesssim H\lesssim 1}  H^{\frac 12}\\
&\quad\qquad\qquad\qquad\qquad\qquad\qquad\quad+\sum_{1\lesssim H\lesssim NN_{min}}  H^{-\frac 12}\bigg)\\
&\quad\quad+ N^{-2s-\frac{1}{2}}\sum_{1\lesssim N_{min}\lesssim N}N_{min}^{s+\frac{n-1}{2}}\bigg(\sum_{H\lesssim 1} \frac{H^{1-\epsilon}}{N^{1-2\epsilon}_{min}}+\sum_{1\lesssim H\lesssim N^2_{min}}\frac{H^{-\epsilon}}{N^{1-2\epsilon}_{min}}  \\
&\quad\qquad\qquad\qquad\qquad\qquad\qquad\quad+\sum_{N^2_{min}\lesssim H\lesssim NN_{min}}  H^{-\frac 12}\bigg)\\
&\quad\lesssim N^{-2s}\sum_{N_{min}\lesssim N^{-1}}N_{min}^{\frac{n}{2}}+N^{-2s-\frac{1}{2}}\left(\sum_{N^{-1}\lesssim N_{min}\lesssim 1}N_{min}^{\frac{n-1}{2}}+ \sum_{1\lesssim N_{min}\lesssim N}N_{min}^{s+\frac{n-3}{2}+2\epsilon}\right)\\
&\quad\lesssim N^{-2s-\frac{1}{2}}\left(1+ \sum_{1\lesssim N_{min}\lesssim N}N_{min}^{s+2\theta+\frac{n-5}{2}}\right),
\endaligned
\end{equation*}
which coincides with the estimate derived for \eqref{lm++-2} in the same scenario. It follows that \eqref{hm++-2} holds true if we impose $s>-1/4$. This concludes the proof of \eqref{hm++-2} and of the entire proposition.

\end{proof}


For the purpose of obtaining LWP results using the framework in our paper, we notice that both \eqref{cucv-xst} and \eqref{uv-xst} require $s>-3/4$ and $s>-1/2$ when $n=2$ and $n=3$, respectively. On the other hand, \eqref{cuv-xst} asks for $s>-1/4$ when either $n=2$ or $n=3$. Hence, a natural question is whether the actual bilinear estimates needed for the fixed point argument (i.e., \eqref{oubvb}-\eqref{oubv}) would be valid for lower values of $s$ than the ones above. We next address comments made earlier that, in our judgement, this is not the case. We take a look at \eqref{oubv} with $\l=1$ chosen for convenience, which, arguing as in the derivation of \eqref{cuv-norm}, is equivalent to
\begin{equation*}
\left\|\frac{|\xi_2|^2\langle\xi_3\rangle^{-s}\langle\tau_3+|\xi_3|^2\rangle^{-\theta}}{\langle\xi_1\rangle^s\langle\tau_1-|\xi_1|^2\rangle^{\theta}\langle\xi_2\rangle^{2-s}\langle\tau_2-|\xi_2|^2\rangle^{1-\theta}}\right\|_{[3,\R^n\times \R]}\lesssim 1.
\end{equation*} 
The corresponding low modulation estimate is given by
\begin{equation}
\aligned
\sum_{N_{max}\sim N_{med}\sim N}\sum_{L_1, L_2, L_3\gtrsim 1}&\frac{\langle N_1\rangle^{-s}N_2^2\langle N_2\rangle^{s-2}\langle N_3\rangle^{-s}}{L_1^\theta L_2^{1-\theta} L_3^{\theta}}\\
&\quad \cdot\left\| X_{N_1, N_2, N_3; L_{max}; L_1, L_2, L_3}\right\|_{[3,\R^n\times \R]}\lesssim 1
\endaligned
\label{o-lm++-2}
\end{equation}
and we consider the coherence scenario where, in addition to \eqref{H-N1-N2}, one has $N_1\sim N_3\gg N_2$ and $H\sim L_2\gg L_1$, $L_3$, $N^2_2$. By applying \eqref{++-2} and $\theta>1/2$, we derive that
\begin{equation*}
\aligned
&\text{(LHS) of \eqref{o-lm++-2}}\\
&\qquad\lesssim N^{-2s-\frac{1}{2}}\sum_{N^{-1}\lesssim N_{min}\lesssim N}\ \sum_{\stackrel{1\lesssim L_{min}\leq L_{med}\leq L_{max}}{N_{min}^2\ll L_{max}\lesssim NN_{min}}}\bigg( \langle N_{min}\rangle^{s-2} N_{min}^{\frac{n+3}{2}}\\
&\qquad\qquad\qquad\qquad\qquad\qquad\qquad\qquad\qquad\qquad\cdot L_{min}^{\frac{1}{2}-\theta}L_{med}^{-\theta}L_{max}^{\theta-\frac{1}{2}}\min\left\{1, \frac{L_{med}}{N_{min}^2}\right\}^{\frac 12}\bigg)\\
&\qquad\lesssim N^{-2s-\frac 12}\sum_{N^{-1}\lesssim N_{min}\lesssim 1}\ \sum_{1\lesssim L_{med}\leq L_{max}\lesssim NN_{min}} N_{min}^{\frac{n+3}{2}}L_{med}^{-\theta}L_{max}^{\theta-\frac{1}{2}}\\
&\qquad \ \ \ + N^{-2s-\frac 12}\sum_{1\lesssim N_{min}\lesssim N}\ \sum_{\stackrel{1\lesssim L_{med}\lesssim N^2_{min}}{N_{min}^2\ll L_{max}\lesssim NN_{min}}} N_{min}^{s+\frac{n-3}{2}}L_{med}^{\frac{1}{2}-\theta}L_{max}^{\theta-\frac{1}{2}}\\
&\qquad \ \ \ + N^{-2s-\frac 12}\sum_{1\lesssim N_{min}\lesssim N}\ \sum_{\stackrel{N^2_{min}\lesssim L_{med}\leq L_{max}}{N_{min}^2\ll L_{max}\lesssim NN_{min}}} N_{min}^{s+\frac{n-1}{2}}L_{med}^{-\theta}L_{max}^{\theta-\frac{1}{2}}\\
&\qquad\lesssim N^{-2s+\theta-1}\left(1+\sum_{1\lesssim N_{min}\lesssim N}N_{min}^{s+\theta+\frac{n-4}{2}}\right),
\endaligned
\end{equation*}
which coincides with the bound obtained in the same setting in the previous proposition. As argued there, one would still need to impose $s>(\theta-1)/2$ (and, thus, $s>-1/4$) for \eqref{o-lm++-2} to hold true.


\section{Alternative method for the summation argument}

In this section, we propose an alternative way to perform the summation component for the proofs of \eqref{lm-gen} and \eqref{hm-gen} (as well as for the ones of \eqref{lm++-2} and \eqref{hm++-2}). It is based on a Python code which streamlines the summation process and, in our opinion, has the potential to be readily adaptable to other similar problems.

In order to explain the idea behind this method, let us discuss first some elementary examples. As in the previous section, we adopt the convention that all variables involved in summations assume only dyadic values. Clearly, for $B$ fixed, one has
\begin{equation*}
\sum_{A \lesssim B} A \sim B.
\end{equation*}
However, when slightly more involved conditional inequalities are introduced in the summation, e.g.,
\begin{equation*}
\sum_{A} \ \sum_{B \lesssim \min\{1, A^{-2}\}} AB,
\end{equation*}
the situation is less straightforward. In fact, for the above sum, one needs to split it into two pieces corresponding to the two possible values of the minimum. As such, it follows that 
\begin{equation*}\aligned
\sum_{A} \ \sum_{B \lesssim \min\{1, A^{-2}\}} AB&=\sum_{A\leq 1} \left(\sum_{B \lesssim 1} B\right)A+\sum_{A>1} \left(\sum_{B \lesssim A^{-2}} B\right)A\\
&\sim \sum_{A\leq 1} A+\sum_{A>1} A^{-1}\sim 1.
\endaligned
\end{equation*}
What we want to stress here is that in order to perform the summation in $B$, we had to split the values of $A$ into two complementary sets. 

When dealing with a summation like the one in \eqref{hm-gen}, which is performed over seven variables (i.e., $(N_i)_{1\leq i\leq 3}$, $(L_i)_{1\leq i\leq 3}$, and $H$), with each one being involved in at least one conditional inequality, the process is obviously much more complex. This is why a computer-assisted analysis makes sense in this type of situation. The way in which we conduct the analysis is as follows:
\begin{enumerate}
\item write the full summation as an iterated summation over each present variable;

\item allow first for the variables to vary independently;

\item let the computer perform the summation;

\item in case the summation yields an infinite result, use one or more conditional inequalities to impose  restrictions on the ranges of the variables and repeat the previous step.

\end{enumerate}

To illustrate  the efficacy of this procedure, we take as a case study the low modulation scenario for \eqref{cucv-xst} with $(n,s,\theta)=(2,-1/2,5/8)$. Hence, the variables involved in \eqref{lm-gen} satisfy the conditional inequalities
\begin{align}
\label{N} N_{max}\sim N_{med}\gtrsim 1\gtrsim N_{min},\\
\label{NN}N_{max}\sim N,\\
\label{L} L_{max}\gg L_{med}\geq L_{min}\gtrsim 1,\\
\label{HNL} H\sim N_{max}^2\sim L_{max},
\end{align}
while, according to \eqref{loc+++},
\begin{equation*}
\left\| X_{N_1, N_2, N_3; H; L_1, L_2, L_3}\right\|_{[3,\R^2\times \R]}\lesssim L_{min}^{\frac 12}N_{max}^{-\frac 12}N_{min}^{\frac{1}{2}} \min\{N_{max}N_{min}, L_{med}\}^{\frac 12}.
\end{equation*}
To be able to work with a summand which is as explicit as possible, we make two assumptions. First, we let
\begin{equation}
\min\{N_{max}N_{min}, L_{med}\}=L_{med}.
\label{NNL}\end{equation}
Secondly, by taking into account \eqref{NLst}, we specialize to the more challenging case when $N_{min}=N_3$ and $L_{max}=L_3$. Thus, the summand has the formula
\[
S =\langle N_{min}\rangle^{-\frac12} N_{min}^{\frac12}\langle N_{med}\rangle^{\frac12}\langle N_{max}\rangle^{\frac12} N_{max}^{-\frac12}L_{min}^{-\frac 18}L_{med}^{-\frac 18}L_{max}^{-\frac 38}.
\]

This is the moment when we initiate the procedure described above, for which the first iteration trivially yields that
\[\sum_{N_{max} = 0}^{\infty}\ \sum_{N_{med} = 0}^{\infty}\ \sum_{N_{min} = 0}^{\infty}\ \sum_{L_{max} = 0}^{\infty}\ \sum_{L_{med} = 0}^{\infty} \ \sum_{L_{min} = 0}^{\infty}\ \sum_{H = 0}^{\infty} S=\infty.\]
Next, we implement \eqref{N} and \eqref{L} jointly with $H\sim L_{max}$ to infer that
\[
S\sim  N_{min}^{\frac12} N_{max}^{\frac12}L_{min}^{-\frac 18}L_{med}^{-\frac 18}L_{max}^{-\frac 38}
\]
and write the summation as
\[\sum_{N_{max} = 2}^{\infty}\ \sum_{N_{med} = \frac{N_{max}}{2}}^{N_{max}}\ \sum_{N_{\min} = 0}^{1}\ \sum_{L_{\max} = 8}^{\infty}\ \sum_{L_{med} = 1}^{\frac{L_{max}}{8}}\ \sum_{L_{min} = 1}^{L_{med}}\ \sum_{H = \frac{L_{max}}{2}}^{2L_{\max}} S.\]
However, another iteration of the third step in the procedure still produces an infinite sum. Following this, we use \eqref{HNL} and \eqref{NNL} to argue that  
$N_{max}N_{min}$ is a better upper bound for $L_{med}$ than $L_{max}/8$. Since $L_{med}\geq 1$, this change also brings about $N_{max}^{-1}$ and $N_{max}N_{min}$ as new, improved lower bounds for $N_{min}$ and $L_{max}$. Consequently, the summation takes the form 
\[\sum_{N_{max} = 2}^{\infty}\ \sum_{N_{med} = \frac{N_{max}}{2}}^{N_{max}}\ \sum_{N_{\min} = N_{max}^{-1}}^{1}\ \sum_{L_{\max} = N_{max}N_{min}}^{\infty}\ \sum_{L_{med} = 1}^{N_{max}N_{min}}\ \sum_{L_{min} = 1}^{L_{med}}\ \sum_{H = \frac{L_{max}}{2}}^{2L_{\max}} S.\]
Unfortunately, by running again the computation step, we obtain infinity for an answer. Finally, if we rely on   the unused part of \eqref{HNL} (i.e., $L_{max}\sim N_{max}^2$), we can modify, with better lower and upper bounds, the sums with respect to $L_{max}$ and $H$. Hence, we are dealing with
\[\sum_{N_{max} = 2}^{\infty}\ \sum_{N_{med} = \frac{N_{max}}{2}}^{N_{max}}\ \sum_{N_{\min} = N_{max}^{-1}}^{1}\ \sum_{L_{\max} = \frac{N^2_{max}}{2}}^{2N^2_{max}}\ \sum_{L_{med} = 1}^{N_{max}N_{min}}\ \sum_{L_{min} = 1}^{L_{med}}\ \sum_{H = \frac{N^2_{max}}{4}}^{4N^2_{\max}} S\]
and another iteration of the third step in our procedure yields a result which is both finite and comparable to $1$. It is worth noticing that we did not make use of \eqref{NN} in the process.

As final comments, let us say that our code is easily adapted to cover the summation arguments for the other types of bilinear estimates proved by Tao in \cite{T-01} (e.g., bounds related to the KdV and wave equations). Moreover, we see no reason not to believe that it can accommodate even general multilinear estimates involving dyadic decompositions.  

\bibliographystyle{amsplain}
\bibliography{bousbib-2}

\end{document}